\documentclass[a4paper,11pt,leqno]{smfart}

\usepackage{amssymb,amsmath,enumerate,verbatim,mathrsfs,graphics,graphicx,mathptm,float,mathtools}
\usepackage[T1]{fontenc}
\usepackage[latin1]{inputenc}
\usepackage[french,english]{babel}
\usepackage[all]{xy}
\usepackage[pdfstartview=FitH,
            colorlinks=true,
            linkcolor=blue,
            urlcolor=blue,
            citecolor=red,
            bookmarks,
            bookmarksopen=true,
            bookmarksnumbered=true]{hyperref}

\usepackage{cleveref}
\theoremstyle{plain}

\theoremstyle{definition}

\theoremstyle{remark}

\theoremstyle{plain}
\newtheorem{thmsec}{Theorem}[section]
\newtheorem{thm}[thmsec]{Theorem}

\newtheorem{lem}[thmsec]{Lemma}
\newtheorem{cor}[thmsec]{Corollary}

\theoremstyle{definition}
\newtheorem{defin}[thmsec]{Definition}

\theoremstyle{remark}
\newtheorem{rem}[thmsec]{Remark}

\newtheorem{eg}[thmsec]{Example}

\def\og{\leavevmode\raise.3ex\hbox{$\scriptscriptstyle\langle\!\langle$~}}
\def\fg{\leavevmode\raise.3ex\hbox{~$\!\scriptscriptstyle\,\rangle\!\rangle$}}

\addtocounter{section}{0}             
\numberwithin{equation}{section}       

\newcommand{\N}{\mathbb{N}}

\newcommand{\R}{\mathbb{R}}
\newcommand{\C}{\mathbb{C}}
\newcommand{\sph}{\mathbb{P}^{1}_{\mathbb{C}}}
\newcommand{\pp}{\mathbb{P}^{2}_{\mathbb{C}}}
\newcommand{\pd}{\mathbb{\check{P}}^{2}_{\mathbb{C}}}

\newcommand\Sing{\mathrm{Sing}}

\newcommand\Leg{\mathrm{Leg}}

\newcommand\ItrH{\mathrm{I}_{\mathcal{H}}^{\hspace{0.2mm}\mathrm{tr}}}

\newcommand\radH{\Sigma_{\mathcal{H}}^{\mathrm{rad}}}
\newcommand\radHd{\check{\Sigma}_{\mathcal{H}}^{\mathrm{rad}}}
\newcommand\F{\mathcal{F}}
\newcommand\W{\mathcal{W}}
\newcommand\G{\mathcal{G}}

\newcommand\Gunderline{{\mspace{2mu}\underline{\mspace{-2mu}\mathcal{G}\mspace{-2mu}}\mspace{2mu}}}

\begin{document}
\title[Holomorphy of the curvature of smooth planar webs and homogeneous foliations]{A criterion for the holomorphy of the curvature of smooth planar webs and applications to dual webs of homogeneous foliations on $\mathbb{P}^{2}_{\mathbb{C}}$}

\date{\today}

\author{Samir \textsc{Bedrouni}}

\address{Facult\'e de Math\'ematiques, USTHB, BP $32$, El-Alia, $16111$ Bab-Ezzouar, Alger, Alg\'erie}
\email{sbedrouni@usthb.dz}

\author{David \textsc{Mar\'{\i}n}}

\thanks{This work has been partially funded by the Ministry of Science, Innovation and Universities of Spain through the grants PGC2018-095998-B-I00 and PID2021-125625NB-I00, by the Agency for Management of University and Research Grants of Catalonia through the grants 2017SGR1725 and 2021SGR01015 and by the Spanish State Research Agency, through the Severo Ochoa and María de Maeztu Program for Centers and Units of Excellence in R\&D (CEX2020-001084-M)}

\address{Departament de Matem\`{a}tiques, Edifici Cc, Universitat Aut\`{o}noma de Barcelona, 08193 Cerdanyola del Vall\`{e}s (Barcelona), Spain. Centre de Recerca Matem\`{a}tica, Edifici Cc, Campus de Bellaterra, 08193 Cerdanyola del Vall\`{e}s (Barcelona), Spain}

\email{david.marin@uab.cat}

\selectlanguage{english}
\maketitle{}

\begin{abstract}
Let $d\geq3$ be an integer. For a holomorphic $d$-web $\mathcal{W}$ on a complex surface $M$, smooth along an irreducible component $D$ of~its~discriminant $\Delta(\mathcal{W}),$ we establish an effective criterion for the holomorphy of the curvature of $\mathcal{W}$ along $D,$ generalizing results on decomposable webs due to \textsc{Mar\'{\i}n}, \textsc{Pereira} and \textsc{Pirio}. As an application, we~deduce a complete characterization for the holomorphy of the curvature of the \textsc{Legendre} transform (dual web) $\mathrm{Leg}\mathcal{H}$ of a homogeneous foliation $\mathcal{H}$ of degree $d$ on $\mathbb{P}^{2}_{\mathbb{C}},$ generalizing some of our previous results. This~then allows us to study the flatness of the $d$-web $\mathrm{Leg}\mathcal{H}$ in the particular case where the foliation $\mathcal{H}$ is \textsc{Galois}. When the \textsc{Galois} group of $\mathcal{H}$ is cyclic, we show that $\mathrm{Leg}\mathcal{H}$ is flat if and only if $\mathcal{H}$ is given, up to linear conjugation, by one of the two 1-forms $\omega_1^{\hspace{0.2mm}d}=y^d\mathrm{d}x-x^d\mathrm{d}y$, $\omega_2^{\hspace{0.2mm}d}=x^d\mathrm{d}x-y^d\mathrm{d}y.$ When the \textsc{Galois} group of $\mathcal{H}$ is non-cyclic, we~obtain that $\mathrm{Leg}\mathcal{H}$ is always flat.

\end{abstract}

\bigskip

\section*{Introduction}

\noindent A (regular) $d$-web $\W$ on $(\mathbb{C}^2,0)$ is the data of a family $\{\F_1,\F_2,\ldots,\F_d\}$ of regular holomorphic foliations on~$(\mathbb{C}^2,0)$ which are pairwise transverse at the origin. We write  $\mathcal{W}=\mathcal{F}_{1}\boxtimes\cdots\boxtimes\mathcal{F}_{d}.$
\smallskip

\noindent A (global) $d$-web on a complex surface $M$ is given in a local chart $(x,y)$ by an implicit differential equation $F(x,y,y')=0$, where $F(x,y,p)=\sum_{i=0}^{d}a_{i}(x,y)p^{d-i}$ is a (reduced) polynomial in $p$ of degree~$d$, having analytic coefficients $a_i$, with $a_0$ not identically zero. In a neighborhood of every point $z_{0}=(x_{0},y_{0})$ such that $a_{0}(x_{0},y_{0})\Delta(x_{0},y_{0})\neq 0$, where $\Delta(x,y)$ is the $p$-discriminant of $F$, the integral curves of this equation define a regular $d$-web on $(\mathbb{C}^2,z_{0}).$
\smallskip

\noindent To every $d$-web $\W$ on $M$ with $d\geq 3$ we can associate a meromorphic $2$-form with poles along the discriminant $\Delta(\W)$, called the curvature of $\W$ and denoted by $K(\W)$, \emph{see}~\S\ref{subsec:courbure-platitude}. A web with zero curvature is called flat. When $M=\pp$ the flatness of a web $\W$ on $\pp$ is characterized by the holomorphy of its curvature $K(\W)$ along the generic points of $\Delta(\W)$.

\noindent In $2008$  \textsc{Pereira} and \textsc{Pirio} \cite[Theorem~7.1]{PP08} established a result on the holomorphy of the curvature of~a~completely~decomposable $d$-web $\mathcal{W}=\mathcal{F}_{1}\boxtimes\cdots\boxtimes\mathcal{F}_{d}.$ In $2013$ \textsc{Mar\'{\i}n} and \textsc{Pereira} \cite[Theorem~1]{MP13} extended~this~result to decomposable webs of the form $\W=\W_2\boxtimes\W_{d-2},$ {\it i.e.} which are the superposition of a $2$-web $\W_2$ and a $(d-2)$-web $\W_{d-2}$. In this paper we establish an effective criterion (Theorem~\ref{thm-critere-holomorphie-kw}) for the holomorphy of the curvature of a $d$-web $\W$ defined on a complex surface and smooth along an irreducible component of its discriminant $\Delta(\W),$ generalizing these two results (\emph{see} Corollary~\ref{cor:critere-holomorphie-kw} and Remark~\ref{rem:critere-du-barycentre}).
\smallskip

\noindent We are then interested in the foliations on $\pp$ which are {\sl homogeneous}, {\it i.e.} which are invariant by homotheties. In \cite[Section~3]{BM18Bull} we studied, for a homogeneous foliation $\mathcal H$ of degree $d\geq3$ on $\pp,$ the problem of the flatness of its \textsc{Legendre} transform (its dual web) $\Leg\mathcal{H}$; it is a $d$-web on the dual projective plane $\pd$ whose leaves are the tangent lines to the leaves of $\mathcal H,$ \emph{see}~\S\ref{sec:holomorphie-courbure-LegH}. Theorem~\ref{thm-critere-holomorphie-kw} allows us to establish, for such a foliation $\mathcal H,$ a complete characterization (Theorem~\ref{thm:holomorphie-courbure-homogene}) of the holomorphy of the curvature of the $d$-web $\mathrm{Leg}\mathcal{H}$ along an irreducible component of the discriminant $\Delta(\mathrm{Leg}\mathcal{H})$, generalizing our results in~\cite[Theorems~3.5~and~3.8]{BM18Bull} (\emph{see} Corollary~\ref{cor:holomorphie-droite-inflex-nu-1} and Remark~\ref{rem:holomorphie-droite-inflex-minimale-maximale}).
\smallskip

\noindent We finally focus on the particular case of a homogeneous foliation $\mathcal H$ of degree $d\geq 3$ on $\pp$ which is \textsc{Galois} in the sense of \cite[Definition~6.16]{BFMN16}, \emph{see}~\S\ref{sec:feuill-homog-Galois-plat}. When the \textsc{Galois} group of $\mathcal{H}$ is cyclic, we prove~that~$\mathrm{Leg}\mathcal{H}$~is~flat~if and only if, up to linear conjugation, $\mathcal{H}$ is given by one of the two $1$-forms $\omega_1^{\hspace{0.2mm}d}=y^d\mathrm{d}x-x^d\mathrm{d}y$, $\omega_2^{\hspace{0.2mm}d}=x^d\mathrm{d}x-y^d\mathrm{d}y.$ When the \textsc{Galois} group of  $\mathcal{H}$ is non-cyclic, we show that $\Leg\mathcal{H}$ is always~flat, \emph{see}~Theorem~\ref{thm:feuill-homog-Galois-plat}.


\section{Preliminaries}

\subsection{Webs}

Let $d\geq1$ be a integer. A {\sl (global) $d$-web} $\mathcal{W}$ on a complex surface $M$ is given by an open~covering $(U_{i})_{i\in I}$ of $M$ and a collection of $d$-symmetric $1$-forms $\omega_{i}\in \mathrm{Sym}^{d}\Omega^{1}_{M}(U_{i})$, with isolated zeros, satisfying:
\begin{itemize}
\item [($\mathfrak{a}$)] there exists $g_{ij}\in \mathcal{O}^{*}_{M}(U_{i}\cap U_{j})$ such that $\omega_i$ coincides with $g_{ij}\omega_j$ on $U_i\cap U_j$;
\item [($\mathfrak{b}$)] for every generic point $m$ of $U_i$, $\omega_i(m)$ factors as the product of $d$ pairwise linearly~independent~$1$-forms.
\end{itemize}

\noindent The {\sl discriminant} $\Delta(\mathcal{W})$ of $\mathcal{W}$ is the divisor on $M$ defined locally by $\Delta(\omega_i)=0$, where $\Delta(\omega_i)$ is the discriminant of the $d$-symmetric $1$-form~$\omega_{i}\in\mathrm{Sym}^{d}\Omega^{1}_{M}(U_{i})$, \emph{see} \cite[Chapter~1, \S 1.3.4]{PP15}. The support of $\Delta(\mathcal{W})$ consists of the points of~$M$ which do not satisfy condition ($\mathfrak{b}$). When~$d=1$ this condition is always satisfied and we recover the usual definition of a holomorphic foliation $\F$ on $M$.

\noindent The {\sl tangent locus} $\mathrm{T}_m\W$ of $\W$ at a point $m\in U_i\setminus\Delta(\W)$ is the union of the $d$ kernels at $m$ of the linear factors of $\omega_i(m).$

\noindent A global $d$-web $\W$ on $M$ is said to be {\sl decomposable} if there are global webs $\mathcal{W}_{1},\mathcal{W}_{2}$ on $M$ sharing no common subwebs such that $\W$ is the superposition of $\mathcal{W}_{1}$ and $\mathcal{W}_{2}$; we then write $\mathcal{W}=\mathcal{W}_{1}\boxtimes\mathcal{W}_{2}.$ Otherwise $\W$  is said to be {\sl irreducible}. We say that $\mathcal{W}$ is {\sl completely decomposable} if there exist global foliations $\mathcal{F}_{1},\ldots,\mathcal{F}_{d}$ on $M$ such that $\mathcal{W}=\mathcal{F}_{1}\boxtimes\cdots\boxtimes\mathcal{F}_{d}.$ For more details on this subject, we refer to~\cite{PP15}.

\subsection{Characteristic surface of a web}

Let $\W$ be a  holomorphic $d$-web on a complex surface $M.$ Let $\widetilde{M}=\mathbb{P}\mathrm{T}^{*}M$ be the projectivization of the cotangent bundle of $M$; the {\sl characteristic surface} of $\W$ is the surface $S_\W\subset\widetilde{M}$ defined by
\[
S_\W:=\overline{\Big\{(m,[\eta])\in \widetilde{M}\hspace{1mm}\big\vert\hspace{1mm} m\in M\setminus\Delta(\W),\hspace{1mm}\ker\eta\subset\mathrm{T}_{m}\W\Big\}}.
\]
We will give a local expression of this surface. First of all, let us consider a local coordinate system~$(x,y)$~on~an open subset $U$ of $M$. Denote by $\pi\hspace{1mm}\colon \widetilde{M}\to M$ the natural projection. We define a coordinate system on the open set  $\pi^{-1}(U),$ by denoting by $(x,y,[p:q])$ the coordinates of the point $(m,[q\mathrm{d}y-p\mathrm{d}x])\in\pi^{-1}(U),$ where $(x,y)$ are the local coordinates of $m$ in $U.$ If $\W$ is given on $U$ by the $d$-symmetric $1$-form $\omega=\sum_{i=0}^{d}a_i(x,y)(\mathrm{d}x)^{i}(\mathrm{d}y)^{d-i}$, with $a_i\in\mathcal{O}_{M}(U)$, then
$$S_\W\cap \pi^{-1}(U)=\{(x,y,[p:q])\in\widetilde{M}\hspace{1mm}\vert\hspace{1mm}\widetilde{F}(x,y,p,q)=0\},$$
where $\widetilde{F}(x,y,p,q)=\sum_{i=0}^{d}a_i(x,y)p^{d-i}q^{i}.$

\noindent In the sequel we will work in the affine chart $(U_q,(x,y,p))$ defined by $U_q:=\pi^{-1}(U)\setminus\{q=0\}$ and $p:=[p:1].$ Setting $F(x,y,p):=\widetilde{F}(x,y,p,1)=\sum_{i=0}^{d}a_i(x,y)p^{d-i},$ we have
\begin{align*}
S_\W\cap U_q=\{(x,y,p)\in\widetilde{M}\hspace{1mm}\vert\hspace{1mm}F(x,y,p)=0\}.
\end{align*}
We will denote by  $\pi_\W\hspace{1mm}\colon S_\W\to M$ the restriction of $\pi$ to $S_\W$. Let us introduce the following definition which will be useful later.

\begin{defin}
With the above notations, let $D$ be an irreducible component of the discriminant $\Delta(\W).$ We will say that $\W$ is \textsl{smooth along $D$} if for every generic point $m$ of $D,$ the characteristic surface $S_\W$ of~$\W$ is smooth at every point of the fiber $\pi_{\W}^{-1}(m).$
\end{defin}

\begin{eg}
On $M=\C^2,$ the $2$-web $\W$ given by $\omega=(y^2-x)\mathrm{d}y^2+2\,x\hspace{0.2mm}\mathrm{d}x\mathrm{d}y-x\hspace{0.2mm}\mathrm{d}x^2$ has discriminant $\Delta(\W)=4xy^2$ and its characteristic surface $S_\W$ has equation $F(x,y,p):=(y^2-x)p^2+2\,xp-x=0.$ Note~that~$\W$ is smooth along the irreducible component  $D_1:=\{x=0\}\subset\Delta(\W).$ Indeed, the fiber $\pi_{\W}^{-1}(m)$ over a generic point $m=(0,y)\in D_1$ is reduced to the point $\widetilde{m}=(0,y,0),$ and the surface $S_\W$ is smooth at $\widetilde{m},$ because  $\partial_xF\big(0,y,0\big)=-1\neq0.$ However,~$\W$~is not smooth along the irreducible component $D_2:=\{y=0\}\subset\Delta(\W),$ because, for every generic point $m=(x,0)\in D_2,$ we have $\pi_{\W}^{-1}(m)=\{(x,0,1)\}$\, and\, $\partial_xF\big(x,0,1\big)\equiv\partial_yF\big(x,0,1\big)\equiv\partial_pF\big(x,0,1\big)\equiv0.$
\end{eg}

\subsection{Fundamental form, curvature and flatness of a web}\label{subsec:courbure-platitude}

We recall here the definitions of the fundamental form and the curvature of a $d$-web $\W.$ Let us first suppose that $\W$ is a germ of completely decomposable $d$-web on $(\mathbb{C}^{2},0)$, $\mathcal{W}=\mathcal{F}_{1}\boxtimes\cdots\boxtimes\mathcal{F}_{d}.$ For each $1\leq i\leq d,$ let $\omega_i$ be a $1$-form with at most an isolated singularity at $0$ defining the foliation $\F_i$. According to \cite{PP08}, for every triple $(r,s,t)$ with $1\leq r<s<t\leq d,$ we~define $\eta_{rst}=\eta(\mathcal{F}_{r}\boxtimes\mathcal{F}_{s}\boxtimes\mathcal{F}_{t})$ as the unique meromorphic  $1$-form satisfying the following equalities:
\begin{equation}\label{equa:eta-rst}
{\left\{\begin{array}[c]{lll}
\mathrm{d}(\delta_{st}\,\omega_{r}) &=& \eta_{rst}\wedge\delta_{st}\,\omega_{r}\\
\mathrm{d}(\delta_{tr}\,\omega_{s}) &=& \eta_{rst}\wedge\delta_{tr}\,\omega_{s}\\
\mathrm{d}(\delta_{rs}\,\omega_{t}) &=& \eta_{rst}\wedge\delta_{rs}\,\omega_{t}
\end{array}
\right.}
\end{equation}
where $\delta_{ij}$ denotes the function defined by the relation $\omega_{i}\wedge\omega_{j}=\delta_{ij}\,\mathrm{d}x\wedge\mathrm{d}y.$ We call {\sl  fundamental form} of the web $\mathcal{W}=\mathcal{F}_{1}\boxtimes\cdots\boxtimes\mathcal{F}_{d}$ the $1$-form
\begin{equation}\label{equa:eta}
\hspace{7mm}\eta(\mathcal{W})=\eta(\mathcal{F}_{1}\boxtimes\cdots\boxtimes\mathcal{F}_{d})=\sum_{1\le r<s<t\le d}\eta_{rst}.
\end{equation}
We can easily verify that $\eta(\mathcal{W})$ is a meromorphic $1$-form with poles along the discriminant $\Delta(\mathcal{W})$ of $\mathcal{W},$ and~that it is well-defined up to addition of a closed logarithmic $1$-form $\dfrac{\mathrm{d}g}{g}$ with $g\in\mathcal{O}^*(\mathbb{C}^{2},0)$ (\emph{cf.} \cite{Rip05,BM18Bull}).
\smallskip

\noindent Now, if $\mathcal{W}$ is an arbitrary $d$-web on a complex surface $M$, then we can transform it into a completely decomposable $d$-web by taking its pull-back by a suitable ramified \textsc{Galois} covering. The invariance of the~fundamental form of this new web by the action of the \textsc{Galois} group allows us to descend it to a global meromorphic~$1$-form $\eta(\mathcal{W})$ on $M$, with poles along the discriminant of $\W$ (\emph{see} \cite{MP13}).
\smallskip

\noindent The {\sl curvature} of the web $\mathcal{W}$ is by definition the $2$-form
\begin{align*}
&K(\mathcal{W})=\mathrm{d}\,\eta(\mathcal{W}).
\end{align*}
\noindent It is a meromorphic $2$-form with poles along the  discriminant $\Delta(\mathcal{W}),$ canonically associated to $\mathcal{W}$; more precisely, for any dominant holomorphic map  $\varphi,$ we have $K(\varphi^{*}\mathcal{W})=\varphi^{*}K(\mathcal{W}).$
\smallskip

\noindent A $d$-web $\mathcal{W}$ is called {\sl flat} if its curvature  $K(\mathcal{W})$ vanishes identically.
\smallskip

\noindent Note that a $d$-web $\mathcal{W}$ on $\mathbb{P}^{2}_{\mathbb{C}}$ is flat if and only if its curvature is holomorphic over the generic points of the irreducible components of  $\Delta(\mathcal{W})$. This follows from the definition of $K(\mathcal{W})$ and the fact that there are no holomorphic $2$-forms on $\mathbb{P}^{2}_{\mathbb{C}}$ other than the zero $2$-form.


\section{Criterion for the holomorphy of the curvature of smooth webs}
\vspace{2mm}

\noindent In this section we propose to establish the following theorem.
\begin{thm}\label{thm-critere-holomorphie-kw}
{\sl Let $\W$ be a holomorphic $d$-web on a complex surface $M$ and let $D$ be an irreducible component~of the discriminant  $\Delta(\W).$ Assume that $\W$ is smooth along $D$. Then the fundamental form~$\eta(\W)$~has~simple poles along $D$. More precisely, choose a local coordinate system $(x,y)$ on $M$ such that $D=\{y=0\}$ and let $F(x,y,p)=0$, $p=\frac{\mathrm{d}y}{\mathrm{d}x},$ be an implicit differential equation defining $\W.$ Write  $F(x,0,p)=a_{0}(x)\prod\limits_{\alpha=1}^{n}(p-\varphi_{\alpha}(x))^{\nu_{\alpha}}$ with $\varphi_{\alpha}\not\equiv\varphi_{\beta}$ if $\alpha\neq\beta.$
Then the $1$-form
\[\eta(\W)-\frac{1}{6y}\sum\limits_{\alpha=1}^{n}\big(\nu_{\alpha}-1\big)\Big(\psi_{\alpha}(x)\big(\mathrm{d}y-\varphi_{\alpha}(x)\mathrm{d}x\big)
+\big(\nu_{\alpha}-2\big)\mathrm{d}y\Big)\]
is holomorphic along $D=\{y=0\}$, where~$\psi_{\alpha}$~is a function of the coordinate  $x$ defined,
for all $\alpha\in\{1,\ldots,n\}$ such that $\nu_{\alpha}\geq2,$ by
\begin{align*}
\psi_{\alpha}(x)=\frac{1}{\nu_{\alpha}}
\left[
(\nu_{\alpha}-2)\left(d-\varphi_{\alpha}(x)\dfrac{\partial_{p}\partial_{y}F\big(x,0,\varphi_{\alpha}(x)\big)}{\partial_{y}F\big(x,0,\varphi_{\alpha}(x)\big)}\right)
-2(\nu_{\alpha}+1)\hspace{-3.5mm}\sum\limits_{\hspace{3.5mm}\beta=1,\beta\neq\alpha}^{n}\frac{\nu_{\beta}\varphi_{\beta}(x)}{\varphi_{\alpha}(x)-\varphi_{\beta}(x)}
\right].
\end{align*}
In particular, the curvature $K(\W)$ is holomorphic along $D$ if and only if
\[
\sum_{\alpha=1}^{n}(\nu_{\alpha}-1)\varphi_{\alpha}(x)\psi_{\alpha}(x)\equiv0
\qquad\text{and}\qquad
\sum_{\alpha=1}^{n}(\nu_{\alpha}-1)\frac{\mathrm{d}}{\mathrm{d}x}\psi_{\alpha}(x)\equiv0.
\]
}
\end{thm}

\begin{rem}
When the component  $D\subset\Delta(\W)$ is totally invariant by $\W$, the curvature  $K(\W)$ is always holomorphic along $D.$
\end{rem}

\begin{rem}
Assume that  $\nu_{\alpha}=\nu\geq2$ for all $\alpha\in\{1,\ldots,n\}.$ The following assertions hold:

\noindent\textbf{\textit{1.}} If $\nu=2$ (which implies that $d$ is even), then the curvature  $K(\W)$ is always holomorphic along $D.$

\noindent\textbf{\textit{2.}} If $\nu\geq 3$, then the curvature $K(\W)$ is holomorphic along $D$ if and only if
\[
\sum_{\alpha=1}^{n}\varphi_{\alpha}(x)\left(d-\rho_{\alpha}(x)\right)\equiv0
\qquad\text{and}\qquad
\sum_{\alpha=1}^{n}\frac{\mathrm{d}}{\mathrm{d}x}\rho_{\alpha}(x)\equiv0,
\]
where $\rho_{\alpha}(x):=\varphi_{\alpha}(x)\dfrac{\partial_{p}\partial_{y}F\big(x,0,\varphi_{\alpha}(x)\big)}{\partial_{y}F\big(x,0,\varphi_{\alpha}(x)\big)}.$

\noindent Indeed, it suffices to set  $f_{\alpha,\beta}=\frac{\varphi_{\beta}}{\varphi_{\alpha}-\varphi_{\beta}}$ and to note that
\begin{align*}
\hspace{-1.4cm}\sum_{\alpha=1}^{n}
\hspace{-3.5mm}\sum\limits_{\hspace{3.5mm}\beta=1,\beta\neq\alpha}^{n}\varphi_{\alpha}f_{\alpha,\beta}
=\sum_{1\leq\alpha<\beta\leq n}\left(\varphi_{\alpha}f_{\alpha,\beta}+\varphi_{\beta}f_{\beta,\alpha}\right)\equiv0
\end{align*}
and
\begin{align*}
\sum_{\alpha=1}^{n}
\hspace{-3.5mm}\sum\limits_{\hspace{3.5mm}\beta=1,\beta\neq\alpha}^{n}f_{\alpha,\beta}
=\sum_{1\leq\alpha<\beta\leq n}\left(f_{\alpha,\beta}+f_{\beta,\alpha}\right)
\equiv-\binom{{n}}{{2}}\equiv\mathrm{constant}.
\end{align*}
\end{rem}
\smallskip

\noindent The hypothesis of smoothness of $\W$ along the component $D\subset\Delta(\W)$ is essential for the validity of Theorem~\ref{thm-critere-holomorphie-kw}, as the following example shows.
\begin{eg}\label{eg:contre-exemple-Henaut}
Let $M$ be a complex surface and let $\W$ be the $3$-web defined
in local coordinates $(x,y)$ by the differential equation
\begin{align*}
F(x,y,p):=\left(\lambda(x^2-1)p+(x-3)y^\kappa\right)\left(\lambda(x^2-1)p+(x+3)y^\kappa\right)\left(\lambda(x^2-1)p-2xy^\kappa\right)=0,
\end{align*}
where $p=\frac{\mathrm{d}y}{\mathrm{d}x},\hspace{1mm}\kappa\in\N\setminus\{0,1\},\hspace{1mm}\lambda\in\C^*.$ For this web we have
\begin{align*}
&\Delta(\W)=2916\lambda^6(x^2-1)^8y^{6\kappa}
\qquad\text{and}\qquad
\eta(\W)=\frac{7\mathrm{d}(x^2-1)}{3(x^2-1)}+\left(\frac{2\kappa}{y}-\frac{\lambda}{3y^\kappa}\right)\mathrm{d}y.
\end{align*}
We see that $\eta(\W)$ is closed and therefore that $\W$ is flat. Moreover $\eta(\W)$ has poles of order $\kappa>1$ along the component $D:=\{y=0\}\subset\Delta(\W).$ Note that $\W$ is not smooth along $D.$ Indeed, the fiber $\pi_{\W}^{-1}(m)$ over a generic point $m=(x,0)\in D$ consists of the single point $\widetilde{m}=(x,0,0)$ and the surface $S_\W$ is not smooth at $\widetilde{m}$, because $\partial_xF\big(x,0,0\big)\equiv\partial_yF\big(x,0,0\big)\equiv\partial_pF\big(x,0,0\big)\equiv0.$
\end{eg}

\begin{rem}
In~\cite[page~286]{Hen06} the author claimed that the fundamental form of a planar $3$-web~$\W$ has probably at most simple poles along $\Delta(\W)$ and he gave an argument in the particular case where $\W$ is defined by a differential equation of type $a_0(x,y)p^3+a_2(x,y)p+a_3(x,y)=0$, $p=\frac{\mathrm{d}y}{\mathrm{d}x}.$ Example~\ref{eg:contre-exemple-Henaut} is of this type and contradicts the claim and the argument of~\cite[page~286]{Hen06}.
\end{rem}

\begin{cor}\label{cor:critere-holomorphie-kw}
{\sl Let $\W$ be a holomorphic $d$-web on a complex surface $M$ and let $D$ be an irreducible component of the discriminant $\Delta(\W).$ Assume that $\W$ is smooth along $D.$ Fix a local coordinate system $(x,y)$ on $M$ such that $D=\{y=0\}$ and let $F(x,y,p)=0,$ $p=\frac{\mathrm{d}y}{\mathrm{d}x},$ be an implicit differential equation defining~$\W$. Assume moreover that $F(x,0,p)=a_{0}(x)(p-\varphi_{0}(x))^\nu\prod\limits_{\alpha=1}^{d-\nu}(p-\varphi_{\alpha}(x))$ with $\varphi_{\alpha}\neq\varphi_{0}$ for all $\alpha\in\{1,\ldots,d-\nu\}$ and $\varphi_{\alpha}\not\equiv\varphi_{\beta}$ if $\alpha\neq\beta.$ Then the curvature  $K(\W)$ is holomorphic on $D$ if and only if $\varphi_{0}\equiv0$ or $\psi\equiv0,$ where
\begin{align*}
\psi(x)=(\nu-2)\left(d-\varphi_{0}(x)\dfrac{\partial_{p}\partial_{y}F\big(x,0,\varphi_{0}(x)\big)}{\partial_{y}F\big(x,0,\varphi_{0}(x)\big)}\right)
-2(\nu+1)\sum\limits_{\alpha=1}^{d-\nu}\frac{\varphi_{\alpha}(x)}{\varphi_{0}(x)-\varphi_{\alpha}(x)}.
\end{align*}
}
\end{cor}

\begin{rem}\label{rem:critere-du-barycentre}
In a neighborhood of every generic point of $D$, the $d$-web $\W$ decomposes as $\W=\W_{\nu}\boxtimes\W_{d-\nu}$ with
\begin{align*}
& \W_{\nu}\Big|_{D}:\mathrm{d}y-\varphi_{0}(x)\mathrm{d}x=0
\qquad\text{and}\qquad
\W_{d-\nu}\Big|_{D}:\prod_{\alpha=1}^{d-\nu}\left(\mathrm{d}y-\varphi_{\alpha}(x)\mathrm{d}x\right)=0.
\end{align*}
When $\nu=2$, we recover the barycenter criterion, namely Theorem~1~of~\cite{MP13} (\emph{see} also \cite[Theorem~7.1]{PP08}): the~curvature~of~$\W=\W_{2}\boxtimes\W_{d-2}$ is holomorphic on $D$ if and only if $D$ is invariant by $\W_2$ or by the barycenter $\beta_{\W_{2}}(\W_{d-2})$ of $\W_{d-2}$ with respect to $\W_{2}.$ Indeed, on the one hand, the invariance of $D=\{y=0\}$ by $\W_2$ translates into $\varphi_{0}\equiv0.$ On the other hand, the restriction of  $\beta_{\W_{2}}(\W_{d-2})$ to $D$ is given by
\[\mathrm{d}y-\Bigg[\varphi_{0}(x)+\frac{1}{\frac{1}{d-2}\sum\limits_{\alpha=1}^{d-2}\frac{1}{\varphi_{\alpha}(x)-\varphi_{0}(x)}}\Bigg]\mathrm{d}x,\]
or equivalently, by
\begin{small}
\begin{align*}
\sum_{\alpha=1}^{d-2}\frac{1}{\varphi_{0}(x)-\varphi_{\alpha}(x)}\mathrm{d}y+\Bigg[d-2-\varphi_{0}(x)\sum_{\alpha=1}^{d-2}\frac{1}{\varphi_{0}(x)-\varphi_{\alpha}(x)}\Bigg]\mathrm{d}x
=\sum_{\alpha=1}^{d-2}\frac{1}{\varphi_{0}(x)-\varphi_{\alpha}(x)}\mathrm{d}y-\sum_{\alpha=1}^{d-2}\frac{\varphi_{\alpha}(x)}{\varphi_{0}(x)-\varphi_{\alpha}(x)}\mathrm{d}x,
\end{align*}
\end{small}
\hspace{-1mm}so that the invariance of $D$ by $\beta_{\W_{2}}(\W_{d-2})$ is characterized by $\sum_{\alpha=1}^{d-2}\frac{\varphi_{\alpha}}{\varphi_{0}-\varphi_{\alpha}}\equiv0$ and therefore by $\psi\equiv0,$ because $\psi=-6\sum_{\alpha=1}^{d-2}\frac{\varphi_{\alpha}}{\varphi_{0}-\varphi_{\alpha}}.$
\end{rem}

\noindent The proof of Theorem~\ref{thm-critere-holomorphie-kw} uses the following lemma.
\begin{lem}\label{lem:eta-3-tissu}
{\sl The fundamental form of the $3$-web  $\W$ defined by the $1$-forms $\omega_\ell=\mathrm{d}y-\lambda_\ell(x,y)\mathrm{d}x$, $\ell=1,2,3$, is given by
\begin{align*}
&\eta(\W)=\sum_{(i,j,k)\in\langle 1,2,3\rangle}\frac{\partial_y(\lambda_i\lambda_j)-\partial_x\lambda_k}{(\lambda_i-\lambda_k)(\lambda_j-\lambda_k)}(\mathrm{d}y-\lambda_k \mathrm{d}x),
\end{align*}
\text{where} $\langle 1,2,3\rangle:=\{(1,2,3),(3,1,2),(2,3,1)\}.$
}
\end{lem}

\begin{proof}
This follows from a straightforward computation using formula~(\ref{equa:eta-rst}).
\end{proof}

\begin{proof}[\sl Proof of Theorem~\ref{thm-critere-holomorphie-kw}]
In a neighborhood of every generic point $m$ of $D,$ the web $\W$ decomposes as $\W=\boxtimes_{\alpha=1}^{n}\W_{\alpha},$ where~$\W_{\alpha}$~is a $\nu_{\alpha}$-web having a unique slope $p=\varphi_{\alpha}(x)$ along $y=0,$ {\it i.e.} $\W_{\alpha}=\boxtimes_{i=1}^{\nu_{\alpha}}\F_{i}^{\alpha}$ and $\F_{i}^{\alpha}|_{y=0}:\mathrm{d}y-\varphi_{\alpha}(x)\mathrm{d}x=0.$ Then $\eta(\W)=\eta_1+\eta_2+\eta_3,$ where
\begin{align*}
&
\eta_1
=\sum_{\raisebox{2mm}{$\underset{\nu_{\alpha}\geq3}{\underset{\alpha=1}{}}$}}^{n}\sum_{1\le i<j<k\le\nu_{\alpha}}\eta_{ijk}^{\alpha\alpha\alpha},
&&
\eta_2=\sum_{\raisebox{2mm}{$\underset{\nu_{\alpha}\geq2}{\underset{\alpha=1}{}}$}}^{n}
\sum_{1\le i<j\le\nu_{\alpha}}
\sum_{\raisebox{2mm}{$\underset{\beta\neq\alpha}{\underset{\beta=1}{}}$}}^{n}
\sum_{k=1}^{\nu_{\beta}}\eta_{ijk}^{\alpha\alpha\beta},
&&
\eta_3=\sum_{1\le\alpha<\beta<\gamma\le n}
\sum_{\raisebox{2mm}{$\underset{1\le k\le\nu_{\gamma}}{\underset{1\le j\le\nu_{\beta}}{\underset{1\le i\le\nu_{\alpha}}{}}}$}}\eta_{ijk}^{\alpha\beta\gamma},
\end{align*}
and $\eta_{ijk}^{\alpha\alpha\alpha}$, resp. $\eta_{ijk}^{\alpha\alpha\beta},$ resp. $\eta_{ijk}^{\alpha\beta\gamma},$ is the fundamental form of the $3$-subweb $\F_{i}^{\alpha}\boxtimes\F_{j}^{\alpha}\boxtimes\F_{k}^{\alpha},$ resp. $\F_{i}^{\alpha}\boxtimes\F_{j}^{\alpha}\boxtimes\F_{k}^{\beta},$  resp. $\F_{i}^{\alpha}\boxtimes\F_{j}^{\beta}\boxtimes\F_{k}^{\gamma}$, of $\W.$

\noindent If $\alpha<\beta<\gamma$ then $(\varphi_{\alpha}-\varphi_{\beta})(\varphi_{\beta}-\varphi_{\gamma})(\varphi_{\gamma}-\varphi_{\alpha})\not\equiv0$, which implies, thanks to Lemma~\ref{lem:eta-3-tissu}, that the $1$-form $\eta_{ijk}^{\alpha\beta\gamma}$ has no poles along $y=0$; therefore the same is true for the $1$-form $\eta_3.$

\noindent As for $\eta_1$ and $\eta_2,$ let us first fix $\alpha\in\{1,\ldots,n\}$ such that $\nu_{\alpha}\geq2.$ Then $\partial_xF\big(x,0,\varphi_{\alpha}(x)\big)\equiv\partial_pF\big(x,0,\varphi_{\alpha}(x)\big)\equiv0$; the hypothesis of smoothness of $\W$ along $D=\{y=0\}$ implies that $\partial_yF\big(x,0,\varphi_{\alpha}(x)\big)\not\equiv0.$ Put
$z=p-\varphi_{\alpha}(x)$ and \begin{small}$F_{\alpha}(x,y,z):=F(x,y,z+\varphi_{\alpha}(x))=\sum_{k\geq0}F_{\alpha,k}(x,z)y^k$\end{small} with \begin{small}$F_{\alpha,k}\in\C\{x\}[z]$\end{small}. Since \begin{small}$F_{\alpha,1}(x,0)=\partial_yF\big(x,0,\varphi_{\alpha}(x)\big)\not\equiv0$\end{small}, the series  \begin{small}$\Psi(y):=\sum_{k\geq1}F_{\alpha,k}y^k$\end{small} is invertible and its inverse writes as  $\Psi^{-1}(w)=\frac{1}{F_{\alpha,1}}w-\frac{F_{\alpha,2}}{(F_{\alpha,1})^3}w^2+\cdots.$ Moreover, define $U_{\alpha}\in\C\{x\}[z]$ by $F_{\alpha}(x,0,z)=z^{\nu_{\alpha}}U_{\alpha}(x,z)$; note that
\begin{align*}
U_{\alpha}(x,z)=a_{0}(x)\hspace{-3.5mm}\prod\limits_{\hspace{3.5mm}\beta=1,\beta\neq\alpha}^{n}\left(z+\varphi_{\alpha}(x)-\varphi_{\beta}(x)\right)^{\nu_{\beta}}
=\sum_{k=0}^{d-\nu_{\alpha}}U_{\alpha,k}(x)z^k,
\end{align*}
with $\frac{U_{\alpha,1}(x)}{U_{\alpha,0}(x)}=\frac{\partial_zU_{\alpha}(x,0)}{U_{\alpha}(x,0)}
=\hspace{-3.5mm}\sum\limits_{\hspace{3.5mm}\beta=1,\beta\neq\alpha}^{n}\frac{\nu_{\beta}}{\varphi_{\alpha}(x)-\varphi_{\beta}(x)}.$ Writing  $F_{\alpha,1}(x,z)=\sum_{k=0}^{d}G_{\alpha,k}(x)z^k,$ with $G_{\alpha,0}\not\equiv0,$ it follows that in a neighborhood of $(x,0,0),$ the equation $F_{\alpha}(x,y,z)=0$ is equivalent to
\begin{small}
\begin{align*}
y=(\Psi^{-1}(-F_{\alpha,0}))(x,z)
=-z^{\nu_{\alpha}}\frac{U_{\alpha}(x,z)}{F_{\alpha,1}(x,z)}-z^{2\nu_{\alpha}}\frac{F_{\alpha,2}(x,z)(U_{\alpha}(x,z))^2}{(F_{\alpha,1}(x,z))^3}+\cdots
=Y_{\alpha,0}(x)z^{\nu_{\alpha}}+Y_{\alpha,1}(x)z^{\nu_{\alpha}+1}+\cdots=:Y_{\alpha}(x,z),
\end{align*}
\end{small}
\hspace{-1mm}with $Y_{\alpha,0}=-\frac{U_{\alpha,0}}{G_{\alpha,0}}\not\equiv0$ and $Y_{\alpha,1}=\frac{G_{\alpha,1}U_{\alpha,0}-G_{\alpha,0}U_{\alpha,1}}{(G_{\alpha,0})^{2}}$ because $\nu_{\alpha}\geq2.$ Thus, we can write $Y_{\alpha}(x,z)=\left(X_{\alpha}(x,z)\right)^{\nu_{\alpha}}$ with $X_{\alpha}(x,z)=\sum_{k\geq1}X_{\alpha,k}(x)z^k,$ $X_{\alpha,1}=\left(Y_{\alpha,0}\right)^{\frac{1}{\nu_{\alpha}}}\not\equiv0$ and $\frac{X_{\alpha,2}}{X_{\alpha,1}}=\frac{Y_{\alpha,1}}{\nu_{\alpha}Y_{\alpha,0}}.$ Then the series $\Phi(z):=\sum_{k\geq1}X_{\alpha,k}z^k$ is invertible and its inverse is of the form $\Phi^{-1}(w)=\sum_{k\geq1}f_{\alpha,k}w^k$ with $f_{\alpha,k}\in\C\{x\},$ $f_{\alpha,1}=\frac{1}{X_{\alpha,1}}$ and $f_{\alpha,2}=-\frac{X_{\alpha,2}}{(X_{\alpha,1})^3}.$ Therefore, the equality $y=(\Psi^{-1}(-F_{\alpha,0}))(x,z)$ is equivalent to $z=(\Phi^{-1}(y^ {\frac{1}{\nu_{\alpha}}}))(x)$ and therefore to $p=(\Phi^{-1}(y^{\frac{1}{\nu_{\alpha}}}))(x)+\varphi_{\alpha}(x).$ As a result, in a neighborhood of $m,$ the slopes $p_{\hspace{-0.2mm}j}$ ($j=1,\ldots,\nu_{\alpha}$) of $\mathrm{T}_{(x,y)}\W_{\alpha}$~are~given~by
\begin{align*}
p_{\hspace{-0.2mm}j}=\lambda_{\alpha,j}(x,y):=\varphi_{\alpha}(x)+\sum_{k\geq1}f_{\alpha,k}(x)\zeta_{\alpha}^{jk}y^{\frac{k}{\nu_{\alpha}}},\quad\text{where}\hspace{1mm} \mathrm{\zeta_{\alpha}}=\exp(\tfrac{2\mathrm{i}\pi}{\nu_{\alpha}}).
\end{align*}
Note furthermore that
\begin{SMALL}
\begin{align}\label{equa:preuve-thm-A}
\frac{f_{\alpha,2}}{(f_{\alpha,1})^2}
=-\frac{X_{\alpha,2}}{X_{\alpha,1}}
=-\frac{Y_{\alpha,1}}{\nu_{\alpha}Y_{\alpha,0}}
=\frac{1}{\nu_{\alpha}}\left(\frac{G_{\alpha,1}}{G_{\alpha,0}}-\frac{U_{\alpha,1}}{U_{\alpha,0}}\right)
=\frac{1}{\nu_{\alpha}}\left[\left(\frac{\partial_z F_{\alpha,1}}{F_{\alpha,1}}\right)\Big|_{z=0}-\frac{U_{\alpha,1}}{U_{\alpha,0}}\right]
=\frac{1}{\nu_{\alpha}}\left[\left(\frac{\partial_z\partial_{y}F_{\alpha}}{\partial_{y}F_{\alpha}}\right)\Big|_{(y,z)=(0,0)}
-\hspace{-3.5mm}\sum\limits_{\hspace{3.5mm}\beta=1,\beta\neq\alpha}^{n}\frac{\nu_{\beta}}{\varphi_{\alpha}-\varphi_{\beta}}\right].
\end{align}
\end{SMALL}
\hspace{-1.38mm}We will now apply Lemma~\ref{lem:eta-3-tissu} to compute $\eta_{ijk}^{\alpha\alpha\alpha}.$ Setting $w_{\alpha}=y^{\frac{1}{\nu_{\alpha}}}$ we obtain
\begin{SMALL}
\begin{align*}
&
\partial_x\lambda_{\alpha,k}=\varphi_{\alpha}^{'}+f_{\alpha,1}^{'}\zeta_{\alpha}^{k}w_{\alpha}+f_{\alpha,2}^{'}\zeta_{\alpha}^{2k}w_{\alpha}^2+f_{\alpha,3}^{'}\zeta_{\alpha}^{3k}w_{\alpha}^3+\cdots,\\
&
\partial_y(\lambda_{\alpha,i}\lambda_{\alpha,j})=\frac{w_{\alpha}}{\nu_{\alpha}y}
\Big[
\varphi_{\alpha}f_{\alpha,1}(\zeta_{\alpha}^{i}+\zeta_{\alpha}^{j})
+2\Big(\varphi_{\alpha}f_{\alpha,2}(\zeta_{\alpha}^{2i}+\zeta_{\alpha}^{2j})+f_{\alpha,1}^2\zeta_{\alpha}^{i+j}\Big)w_{\alpha}
+3\Big(\varphi_{\alpha}f_{\alpha,3}(\zeta_{\alpha}^{3i}+\zeta_{\alpha}^{3j})+f_{\alpha,1}f_{\alpha,2}(\zeta_{\alpha}^{2i+j}+\zeta_{\alpha}^{i+2j})\Big)w_{\alpha}^2+\cdots
\Big],\\
&
(\lambda_{\alpha,i}-\lambda_{\alpha,k})(\lambda_{\alpha,j}-\lambda_{\alpha,k})=w_{\alpha}^{2}(\zeta_{\alpha}^{i}-\zeta_{\alpha}^{k})(\zeta_{\alpha}^{j}-\zeta_{\alpha}^{k})
\Big[f_{\alpha,1}^2+f_{\alpha,1}f_{\alpha,2}(\zeta_{\alpha}^{i}+\zeta_{\alpha}^{j}+2\zeta_{\alpha}^{k})w_{\alpha}+\cdots
\Big].
\end{align*}
\end{SMALL}
\hspace{-1.38mm}According to Lemma~\ref{lem:eta-3-tissu}, we have $\eta_{ijk}^{\alpha\alpha\alpha}=a_{ijk}(x,y)\mathrm{d}x+b_{ijk}(x,y)\mathrm{d}y,$ where
\begin{SMALL}
\begin{align*}
a_{ijk}\hspace{-2.3cm}&\hspace{2.3cm}=
-\frac{\left(\partial_y(\lambda_{\alpha,i}\lambda_{\alpha,j})-\partial_x\lambda_{\alpha,k}\right)\lambda_{\alpha,k}}{(\lambda_{\alpha,i}-\lambda_{\alpha,k})(\lambda_{\alpha,j}-\lambda_{\alpha,k})}
-\frac{\left(\partial_y(\lambda_{\alpha,k}\lambda_{\alpha,j})-\partial_x\lambda_{\alpha,i}\right)\lambda_{\alpha,i}}{(\lambda_{\alpha,k}-\lambda_{\alpha,i})(\lambda_{\alpha,j}-\lambda_{\alpha,i})}
-\frac{\left(\partial_y(\lambda_{\alpha,i}\lambda_{\alpha,k})-\partial_x\lambda_{\alpha,j}\right)\lambda_{\alpha,j}}{(\lambda_{\alpha,i}-\lambda_{\alpha,j})(\lambda_{\alpha,k}-\lambda_{\alpha,j})}\\
\hspace{-2.3cm}&\hspace{2.3cm}=
-\frac{1}{\nu_{\alpha}y}
\left[
\frac{\varphi_{\alpha}}{f_{\alpha,1}^2}\Big(f_{\alpha,1}^2-\varphi_{\alpha}f_{\alpha,2}\Big)
+2\frac{\varphi_{\alpha}^2}{f_{\alpha,1}^3}\Big(\zeta_{\alpha}^i+\zeta_{\alpha}^j+\zeta_{\alpha}^k\Big)
\Big(f_{\alpha,2}^2-f_{\alpha,1}f_{\alpha,3}\Big)w_{\alpha}+A_{-1}w_{\alpha}^{2}\right]+A_0,
\quad{\fontsize{11}{11pt}\text{with}}\hspace{1mm}A_{-1},A_{0}\in\C\{x,w_{\alpha}\}
\end{align*}
\end{SMALL}
\hspace{-1.38mm}and
\begin{SMALL}
\begin{align*}
b_{ijk}&=
\frac{\partial_y(\lambda_{\alpha,i}\lambda_{\alpha,j})-\partial_x\lambda_{\alpha,k}}{(\lambda_{\alpha,i}-\lambda_{\alpha,k})(\lambda_{\alpha,j}-\lambda_{\alpha,k})}
+\frac{\partial_y(\lambda_{\alpha,k}\lambda_{\alpha,j})-\partial_x\lambda_{\alpha,i}}{(\lambda_{\alpha,k}-\lambda_{\alpha,i})(\lambda_{\alpha,j}-\lambda_{\alpha,i})}
+\frac{\partial_y(\lambda_{\alpha,i}\lambda_{\alpha,k})-\partial_x\lambda_{\alpha,j}}{(\lambda_{\alpha,i}-\lambda_{\alpha,j})(\lambda_{\alpha,k}-\lambda_{\alpha,j})}\\
&=
\frac{1}{\nu_{\alpha}y}
\left[
\frac{1}{f_{\alpha,1}^2}\Big(2f_{\alpha,1}^2-\varphi_{\alpha}f_{\alpha,2}\Big)
+\frac{1}{f_{\alpha,1}^3}\Big(\zeta_{\alpha}^i+\zeta_{\alpha}^j+\zeta_{\alpha}^k\Big)\Big(f_{\alpha,1}^2f_{\alpha,2}-2\varphi_{\alpha}f_{\alpha,1}f_{\alpha,3}+2\varphi_{\alpha}f_{\alpha,2}^2\Big)w_{\alpha}
+B_{-1}w_{\alpha}^{2}\right]+B_0,
\quad{\fontsize{11}{11pt}\text{with}}\hspace{1mm}B_{-1},B_{0}\in\C\{x,w_{\alpha}\}.
\end{align*}
\end{SMALL}
\hspace{-1.38mm}Since $\eta_1=\hspace{-3.5mm}\sum\limits_{\hspace{3.5mm}\alpha=1,\nu_{\alpha}\geq3}^{n}\hspace{-3.5mm}\sum\limits_{\hspace{3.5mm}1\le i<j<k\le\nu_{\alpha}}\eta_{ijk}^{\alpha\alpha\alpha}$ is a uniform and meromorphic $1$-form, it follows that the principal part of the \textsc{Laurent} series of $\eta_1$ at $y=0$ is given by $\frac{\theta_1}{y}$, where
\begin{align*}
\theta_1&=\sum_{\raisebox{2mm}{$\underset{\nu_{\alpha}\geq3}{\underset{\alpha=1}{}}$}}^{n}\binom{{\nu_{\alpha}}}{{3}}
\bigg(-\frac{\varphi_{\alpha}(f_{\alpha,1}^2-\varphi_{\alpha}f_{\alpha,2})}{\nu_{\alpha}f_{\alpha,1}^2}\mathrm{d}x+\frac{2f_{\alpha,1}^2-\varphi_{\alpha}f_{\alpha,2}}{\nu_{\alpha}f_{\alpha,1}^2}\mathrm{d}y\bigg)\\
&=\frac{1}{6}\sum_{\raisebox{2mm}{$\underset{\nu_{\alpha}\geq3}{\underset{\alpha=1}{}}$}}^{n}(\nu_{\alpha}-1)(\nu_{\alpha}-2)
\bigg(\Big(1-\frac{\varphi_{\alpha}f_{\alpha,2}}{f_{\alpha,1}^2}\Big)\Big(\mathrm{d}y-\varphi_{\alpha}\mathrm{d}x\Big)+\mathrm{d}y\bigg).
\end{align*}

\noindent It remains to determine the principal part of the \textsc{Laurent} series of $\eta_2$ at $y=0.$ Again according to Lemma~\ref{lem:eta-3-tissu}, we have $ \eta_{ijk}^{\alpha\alpha\beta}=\widetilde{a}_{ijk}(x,y)\mathrm{d}x+\widetilde{b}_{ijk}(x,y)\mathrm{d}y,$ where
\begin{SMALL}
\begin{align*}
\widetilde{a}_{ijk}\hspace{-1.9cm}&\hspace{1.9cm}=
-\frac{\left(\partial_y(\lambda_{\alpha,i}\lambda_{\alpha,j})-\partial_x\lambda_{\beta,k}\right)\lambda_{\beta,k}}{(\lambda_{\alpha,i}-\lambda_{\beta,k})(\lambda_{\alpha,j}-\lambda_{\beta,k})}
-\frac{\left(\partial_y(\lambda_{\beta,k}\lambda_{\alpha,j})-\partial_x\lambda_{\alpha,i}\right)\lambda_{\alpha,i}}{(\lambda_{\beta,k}-\lambda_{\alpha,i})(\lambda_{\alpha,j}-\lambda_{\alpha,i})}
-\frac{\left(\partial_y(\lambda_{\alpha,i}\lambda_{\beta,k})-\partial_x\lambda_{\alpha,j}\right)\lambda_{\alpha,j}}{(\lambda_{\alpha,i}-\lambda_{\alpha,j})(\lambda_{\beta,k}-\lambda_{\alpha,j})}\\
\hspace{-1.9cm}&\hspace{1.9cm}=\frac{1}{\nu_{\alpha}y}
\left[
\frac{\varphi_{\alpha}\varphi_{\beta}}{\varphi_{\alpha}-\varphi_{\beta}}+
\frac{\left(\zeta_{\alpha}^i+\zeta_{\alpha}^j\right)\left((\varphi_{\alpha}-\varphi_{\beta})f_{\alpha,2}-f_{\alpha,1}^2\right)\varphi_{\alpha}\varphi_{\beta}}
{(\varphi_{\alpha}-\varphi_{\beta})^2f_{\alpha,1}}w_{\alpha}+
\frac{(\nu_{\alpha}+\nu_{\beta})\zeta_{\beta}^k\varphi_{\alpha}^2f_{\beta,1}}{\nu_{\beta}(\varphi_{\alpha}-\varphi_{\beta})^2}w_{\beta}+\cdots
\right]+\widetilde{A}_0,\quad{\fontsize{11}{11pt}\text{with}}\hspace{1mm}\widetilde{A}_{0}\in\C\{x,w_{\alpha},w_{\beta}\}
\end{align*}
\end{SMALL}
\hspace{-1.38mm}and

\begin{SMALL}
\begin{align*}
\widetilde{b}_{ijk}&=
\frac{\partial_y(\lambda_{\alpha,i}\lambda_{\alpha,j})-\partial_x\lambda_{\beta,k}}{(\lambda_{\alpha,i}-\lambda_{\beta,k})(\lambda_{\alpha,j}-\lambda_{\beta,k})}
+\frac{\partial_y(\lambda_{\beta,k}\lambda_{\alpha,j})-\partial_x\lambda_{\alpha,i}}{(\lambda_{\beta,k}-\lambda_{\alpha,i})(\lambda_{\alpha,j}-\lambda_{\alpha,i})}
+\frac{\partial_y(\lambda_{\alpha,i}\lambda_{\beta,k})-\partial_x\lambda_{\alpha,j}}{(\lambda_{\alpha,i}-\lambda_{\alpha,j})(\lambda_{\beta,k}-\lambda_{\alpha,j})}\\
&=-\frac{1}{\nu_{\alpha}y}
\left[
\frac{\varphi_{\beta}}{\varphi_{\alpha}-\varphi_{\beta}}+
\frac{\left(\zeta_{\alpha}^i+\zeta_{\alpha}^j\right)\left(\varphi_{\beta}(\varphi_{\alpha}-\varphi_{\beta})f_{\alpha,2}-\varphi_{\alpha}f_{\alpha,1}^2\right)}{(\varphi_{\alpha}-\varphi_{\beta})^2f_{\alpha,1}}w_{\alpha}+
\frac{\Big((2\nu_{\alpha}+\nu_{\beta})\varphi_{\alpha}-\nu_{\alpha}\varphi_{\beta}\Big)\zeta_{\beta}^kf_{\beta,1}}{\nu_{\beta}(\varphi_{\alpha}-\varphi_{\beta})^2}w_{\beta}+\cdots
\right]+\widetilde{B}_0,\hspace{1.5mm}{\fontsize{11}{11pt}\text{with}}\hspace{1mm}\widetilde{B}_{0}\in\C\{x,w_{\alpha},w_{\beta}\}.
\end{align*}
\end{SMALL}
\hspace{-1.38mm}The $1$-form
$\eta_2=\hspace{-3.5mm}\sum\limits_{\hspace{3.5mm}\alpha=1,\nu_{\alpha}\geq2}^{n}\sum\limits_{1\le i<j\le\nu_{\alpha}}\hspace{-3.5mm}\sum\limits_{\hspace{3.5mm}\beta=1,\beta\neq\alpha}^{n}\sum\limits_{k=1}^{\nu_{\beta}}\eta_{ijk}^{\alpha\alpha\beta}$ being uniform and meromorphic, it follows that the principal part of the \textsc{Laurent} series of $\eta_2$ at $y=0$ is given by $\frac{\theta_2}{y}$, where
\begin{align*}
\theta_2&=
\sum_{\raisebox{2mm}{$\underset{\nu_{\alpha}\geq2}{\underset{\alpha=1}{}}$}}^{n}\binom{{\nu_{\alpha}}}{{2}}
\sum_{\raisebox{2mm}{$\underset{\beta\neq\alpha}{\underset{\beta=1}{}}$}}^{n}
\nu_{\beta}\left(\frac{\varphi_{\alpha}\varphi_{\beta}}{\nu_{\alpha}(\varphi_{\alpha}-\varphi_{\beta})}\mathrm{d}x-
\frac{\varphi_{\beta}}{\nu_{\alpha}(\varphi_{\alpha}-\varphi_{\beta})}\mathrm{d}y\right)\\
&=-\frac{1}{2}
\sum_{\raisebox{2mm}{$\underset{\nu_{\alpha}\geq2}{\underset{\alpha=1}{}}$}}^{n}\left(\nu_{\alpha}-1\right)\left(\mathrm{d}y-\varphi_{\alpha}\mathrm{d}x\right)
\sum_{\raisebox{2mm}{$\underset{\beta\neq\alpha}{\underset{\beta=1}{}}$}}^{n}\frac{\nu_{\beta}\varphi_{\beta}}{\varphi_{\alpha}-\varphi_{\beta}}.
\end{align*}

\noindent As a consequence, the principal part of the \textsc{Laurent} series of $\eta(\W)$ at $y=0$ is given by $\frac{\theta}{y}$, where
\begin{align*}
\theta=\theta_1+\theta_2
=\frac{1}{6}\sum_{\raisebox{2mm}{$\underset{\nu_{\alpha}\geq2}{\underset{\alpha=1}{}}$}}^{n}
(\nu_{\alpha}-1)
\Bigg\{
\Bigg[
(
\nu_{\alpha}-2
)
\Big(
1-\frac{\varphi_{\alpha}f_{\alpha,2}}{f_{\alpha,1}^2}
\Big)
-3\sum_{\raisebox{2mm}{$\underset{\beta\neq\alpha}{\underset{\beta=1}{}}$}}^{n}\frac{\nu_{\beta}\varphi_{\beta}}{\varphi_{\alpha}-\varphi_{\beta}}
\Bigg]
(
\mathrm{d}y-\varphi_{\alpha}\mathrm{d}x
)
+(\nu_{\alpha}-2)\mathrm{d}y
\Bigg\}.
\end{align*}

\noindent Thanks to~(\ref{equa:preuve-thm-A}), the $1$-form $\theta$ can be rewritten as
\begin{Small}
\begin{align*}
\theta
&=\frac{1}{6}\sum_{\raisebox{2mm}{$\underset{\nu_{\alpha}\geq2}{\underset{\alpha=1}{}}$}}^{n}
(\nu_{\alpha}-1)
\Bigg\{
\Bigg[
(
\nu_{\alpha}-2
)
\bigg(1-\frac{\varphi_{\alpha}}{\nu_{\alpha}}
\Big(
\frac{\partial_z\partial_{y}F_{\alpha}}{\partial_{y}F_{\alpha}}
\Big)\Big|_{(y,z)=(0,0)}
\bigg)
+\sum_{\raisebox{2mm}{$\underset{\beta\neq\alpha}{\underset{\beta=1}{}}$}}^{n}
\frac{\nu_{\beta}\left((\nu_{\alpha}-2)\varphi_{\alpha}-3\nu_{\alpha}\varphi_{\beta}\right)}{\nu_{\alpha}\left(\varphi_{\alpha}-\varphi_{\beta}\right)}
\Bigg]
(
\mathrm{d}y-\varphi_{\alpha}\mathrm{d}x
)
+(\nu_{\alpha}-2)\mathrm{d}y
\Bigg\}.
\end{align*}
\end{Small}
\hspace{-1.38mm}Now, we have
\begin{align*}
\sum_{\raisebox{2mm}{$\underset{\beta\neq\alpha}{\underset{\beta=1}{}}$}}^{n}
\frac{\nu_{\beta}\left((\nu_{\alpha}-2)\varphi_{\alpha}-3\nu_{\alpha}\varphi_{\beta}\right)}{\nu_{\alpha}\left(\varphi_{\alpha}-\varphi_{\beta}\right)}
=\frac{1}{\nu_{\alpha}}
\Bigg[
(\nu_{\alpha}-2)(d-\nu_{\alpha})-2(\nu_{\alpha}+1)
\sum_{\raisebox{2mm}{$\underset{\beta\neq\alpha}{\underset{\beta=1}{}}$}}^{n}\frac{\nu_{\beta}\varphi_{\beta}}{\varphi_{\alpha}-\varphi_{\beta}}
\Bigg],
\hspace{1mm}\text{because}\hspace{1mm}
d=\sum_{\beta=1}^{n}\nu_{\beta}.
\end{align*}
\noindent Therefore
\begin{small}
\begin{align*}
\theta
&=\frac{1}{6}\sum_{\raisebox{2mm}{$\underset{\nu_{\alpha}\geq2}{\underset{\alpha=1}{}}$}}^{n}
(\nu_{\alpha}-1)
\Bigg\{\frac{1}{\nu_{\alpha}}
\Bigg[
(
\nu_{\alpha}-2
)
\bigg(d-\varphi_{\alpha}
\Big(
\frac{\partial_z\partial_{y}F_{\alpha}}{\partial_{y}F_{\alpha}}
\Big)\Big|_{(y,z)=(0,0)}
\bigg)
-2(\nu_{\alpha}+1)\sum_{\raisebox{2mm}{$\underset{\beta\neq\alpha}{\underset{\beta=1}{}}$}}^{n}
\frac{\nu_{\beta}\varphi_{\beta}}{\varphi_{\alpha}-\varphi_{\beta}}
\Bigg]
(
\mathrm{d}y-\varphi_{\alpha}\mathrm{d}x
)
+(\nu_{\alpha}-2)\mathrm{d}y
\Bigg\}\\
&=
\frac{1}{6}\sum_{\alpha=1}^{n}
(\nu_{\alpha}-1)
\Big(
\psi_{\alpha}(\mathrm{d}y-\varphi_{\alpha}\mathrm{d}x)+(\nu_{\alpha}-2)\mathrm{d}y
\Big),
\end{align*}
\end{small}
hence the theorem follows.
\end{proof}


\section{Holomorphy of the curvature of the dual web of a homogeneous foliation on $\pp$}\label{sec:holomorphie-courbure-LegH}
\vspace{2mm}

\noindent Following~\cite[Definition 2.1]{BM18Bull} a {\sl homogeneous} foliation $\mathcal{H}$ of degree $d$ on  $\pp$ is given, in a suitable choice of affine coordinates $(x,y)$, by a homogeneous $1$-form $\omega=A(x,y)\mathrm{d}x+B(x,y)\mathrm{d}y,$ where $A,B\in\mathbb{C}[x,y]_d$ and~$\mathrm{gcd}(A,B)=1.$

\noindent The tangent lines to the leaves of  $\mathcal{H}$ are the leaves of a $d$-web on the dual projective plane $\pd$, called the {\sl \textsc{Legendre} transform} (or {\sl dual web}) of $\mathcal{H}$, and denoted by $\Leg\mathcal{H}.$ More precisely, let $(p,q)$ be the affine chart of $\pd$ corresponding to the line $\{y=px-q\}\subset{\mathbb{P}^{2}_{\mathbb{C}}}$; then $\Leg\mathcal{H}$ is given by the implicit differential equation (\emph{see}~\cite{MP13})
\begin{align}\label{equa:LegH}
A(x,px-q)+pB(x,px-q)=0, \qquad \text{with} \qquad x=\frac{\mathrm{d}q}{\mathrm{d}p}.
\end{align}
The \textsc{Gauss} map of $\mathcal H$ is the rational map $\mathcal{G}_{\mathcal{H}}\hspace{1mm}\colon\pp\dashrightarrow \pd$ defined at every regular point $m$ of $\mathcal H$ by $\mathcal{G}_{\mathcal{H}}(m)=\mathrm{T}^{\mathbb{P}}_{m}\mathcal{H},$ where $\mathrm{T}^{\mathbb{P}}_{m}\mathcal{H}$ denotes the tangent line to the leaf of $\mathcal{H}$ passing through $m.$ According to \cite[Lemma~3.2]{BM18Bull}~the~discriminant of $\Leg\mathcal{H}$ decomposes as $$\Delta(\Leg\mathcal{H})=\G_{\mathcal{H}}(\ItrH)\cup\radHd\cup\check{O},$$ where $\ItrH$ is the transverse inflection divisor of $\mathcal{H}$, $\radHd$ is the set of lines dual to the radial singularities of $\mathcal{H}$ and finally $\check{O}$ is the dual line of the origin of the affine chart $(x,y).$ For precise definitions of radial singularities and the inflection divisor of a foliation on $\pp$, we refer to \cite[\S1.3]{BM18Bull}.

\noindent To the homogeneous foliation $\mathcal{H}$ we can also associate the rational map $\Gunderline_{\mathcal{H}}\hspace{1mm}\colon\mathbb{P}^{1}_{\mathbb{C}}\rightarrow \mathbb{P}^{1}_{\mathbb{C}}$ defined by $$\Gunderline_{\mathcal{H}}([y:x])=[-A(x,y):B(x,y)],$$
which allows us to completely determine the divisor $\ItrH$ and the set $\radH$ (\emph{see}~\cite[Section~2]{BM18Bull}):
\begin{itemize}
\item [$\bullet$] $\radH$ consists of $[b:a:0]\in L_{\infty}$ such that $[a:b]\in\mathbb{P}^{1}_{\mathbb{C}} $ is a fixed critical point of $\Gunderline_{\mathcal{H}}$;
\item [$\bullet$] $\ItrH=\prod\limits_{i}T_{i}^{n_i}$, where $T_{i}=(b_i\hspace{0.2mm}y-a_i\hspace{0.2mm}x=0)$ and $[a_i:b_i]\in\mathbb{P}^{1}_{\mathbb{C}}$  is a non-fixed critical point of $\Gunderline_{\mathcal{H}}$ of multiplicity~$n_i.$
\end{itemize}

\noindent We know from \cite[Lemma~3.1]{BFM13} that if the curvature of $\Leg\mathcal{H}$ is holomorphic on $\pd\hspace{-0.3mm}\setminus\hspace{-0.3mm } \check{O},$ then $\Leg\mathcal{H}$~is~flat. The~following~theorem is an effective criterion for the holomorphy of the curvature of $\Leg\mathcal{H}$ along an irreducible component $D$ of $\Delta(\mathrm{Leg}\mathcal{H })\setminus\check{O}.$
\begin{thm}\label{thm:holomorphie-courbure-homogene}
{\sl Let $\mathcal{H}$ be a homogeneous foliation of degree $d\geq3$ on $\pp$ defined by the $1$-form
$$\omega=A(x,y)\mathrm{d}x+B(x,y)\mathrm{d}y,\quad A,B\in\mathbb{C}[x,y]_d, \hspace{2mm}\gcd(A,B)=1.$$ Let $(p,q)$ be the affine chart of $\pd$ associated to the line $\{y=px-q\}\subset{ \mathbb{P}^{2}_{\mathbb{C}}}$ and let $D=\{p=p_0\}$ be an irreducible component of $\Delta(\mathrm{Leg}\mathcal{H} )\setminus\check{O}.$ Write $\Gunderline_{\mathcal H}^{-1}([p_0:1])=\{[a_1:b_1],\ldots,[a_n:b_n]\} $ and denote by $\nu_i$ the ramification index of $\Gunderline_{\mathcal H}$ at the point $[a_i:b_i]\in\mathbb{P}_{\C}^{1}.$ For $ i\in\{1,\ldots,n\}$, define the polynomials $P_i\in\C[x,y]_{d-\nu_i}$ and $Q_i\in\C[x,y]_ {2d-\nu_i-1}$ by
\begin{Small}
\begin{align*}
P_i(x,y;a_i,b_i):=\frac{\left|
\begin{array}{cc}
A(x,y)  &  A(b_i,a_i)
\\
B(x,y)  &  B(b_i,a_i)
\end{array}
\right|}{(b_iy-a_i\hspace{0.2mm}x)^{\nu_i}}
\quad{\fontsize{11}{11pt}\text{and}}\quad
Q_i(x,y;a_i,b_i):=(\nu_i-2)\left(\dfrac{\partial{B}}{\partial{x}}-\dfrac{\partial{A}}{\partial{y}}\right)P_i(x,y;a_i,b_i)+2(\nu_i+1)
\left|\begin{array}{cc}
\dfrac{\partial{P_i}}{\partial{x}} &  A(x,y)
\vspace{2mm}
\\
\dfrac{\partial{P_i}}{\partial{y}} &  B(x,y)
\end{array} \right|.
\end{align*}
\end{Small}
\hspace{-1mm}Then the curvature of  $\mathrm{Leg}\mathcal{H}$ is holomorphic on $D$ if and only if
\begin{align*}
\sum_{i=1}^{n}\left(1-\frac{1}{\nu_{i}}\right)\frac{(p_0\hspace{0.2mm}b_i-a_i)Q_i(b_i,a_i;a_i,b_i)}{P_i(b_i,a_i;a_i,b_i)B(b_i,a_i)}=0.
\end{align*}
}
\end{thm}

\begin{rem}
In particular, if $D\subset\radHd\setminus\G_{\mathcal{H}}(\ItrH),$ or equivalently, if all the critical points of  $\Gunderline_{\mathcal H}$ in the fiber $\Gunderline_{\mathcal H}^{-1}([p_0:1])$ are fixed, then the curvature $K(\mathrm{Leg}\mathcal{H}) $ is always holomorphic on $D$; indeed, we then have $p_0\hspace{0.2mm}b_i-a_i=0$ if $\nu_i\geq2.$
\end{rem}

\noindent Combining this remark with \cite[Lemma~3.1]{BFM13}, we recover Theorem~3.1 of \cite{BM18Bull}: the $d$-web $\Leg\mathcal{H}$ is flat if and only if its curvature $K(\Leg\mathcal{H})$ is holomorphic on $\G_{\mathcal{H}}(\ItrH).$

\begin{rem}\label{rem:indice-ramification-homogene}
Assume that $\nu_{i}=\nu\geq2$ for all $i\in\{1,\ldots,n\}.$ The following assertions hold:
\begin{itemize}
\item [\textbf{\textit{1.}}] When $\nu=2$ (which implies that $d$ is even), the curvature of $\mathrm{Leg}\mathcal{H}$ is always holomorphic~on~$D.$

\item [\textbf{\textit{2.}}] When $\nu\geq3,$ the curvature of $\mathrm{Leg}\mathcal{H}$ is holomorphic on $D$ if and only if
\begin{align*}
\sum_{i=1}^{n}\frac{\big(p_0\hspace{0.2mm}b_i-a_i\big)\big(\partial_{x}B(b_i,a_i)-\partial_{y} A(b_i,a_i)\big)}{B(b_i,a_i)}=0.
\end{align*}
In particular, if the fiber $\Gunderline_{\mathcal H}^{-1}([p_0:1])$ contains a single non-fixed critical point of $\Gunderline_{\mathcal H}$, say $[a: b]$, then
\begin{itemize}
\item [--] either $\Gunderline_{\mathcal H}^{-1}([p_0:1])=\{[a:b]\}$, in which case $\nu=d$;
\item [--] or $\#\Gunderline_{\mathcal H}^{-1}([p_0:1])=2$, in which case $d$ is necessarily even, $d=2k,$ and $\nu=k.$
\end{itemize}
In both cases, the curvature of $\mathrm{Leg}\mathcal{H}$ is holomorphic on $D$ if and only if the $2$-form $\mathrm{d}\omega$ vanishes on the line~$T=(by-ax=0),$ which is the transverse inflection line of $\mathcal{H}$ associated to the non-fixed critical point $[a:b]$ of $\Gunderline_{\mathcal H}.$
\end{itemize}
\end{rem}
\bigskip

\begin{eg}
Consider the homogeneous foliation $\mathcal{H}$ of even degree $2k\geq4$ on $\pp$ defined by the $1$-form
\begin{align*}
\hspace{3.3cm}&\omega=y^k(y-x)^k\mathrm{d}x+(y-\lambda\,x)^k(y-\mu\,x)^k\mathrm{d}y,&& \text{where }\lambda,\mu\in\C\setminus\{0,1\}.
\end{align*}
In the affine chart $(p,q)$ of $\pd$ associated to the line $\{y=px-q\}\subset{\pp},$ the web $\Leg\mathcal{H}$ is implicitly described by the equation
\begin{align*}
(px-q)^k(px-q-x)^k+p(px-q-\lambda\,x)^k(px-q-\mu\,x)^k=0, \qquad \text{with} \qquad x=\frac{\mathrm{d}q}{\mathrm{d}p}.
\end{align*}
We see that $\Leg\mathcal{H}$ has a single slope $x=-q$ along $D:=\{p=0\}$, so that $D\subset\Delta(\Leg \mathcal{H}).$ Moreover, the map $\Gunderline_{\mathcal{H}}$ is given, for any $[x:y]\in\mathbb{P}^{1}_{ \mathbb{C}},$ by
$$\Gunderline_{\mathcal H}([x:y])=[-x^k(x-y)^k:(x-\lambda\,y)^k(x-\mu\,y)^k ].$$
In particular, the fiber $\Gunderline_{\mathcal H}^{-1}([0:1])$ consists of the two points $[0:1]$ and $[1:1]$: the point $[0:1]$ (resp.~$[1:1]$) is critical and fixed (resp.~non-fixed) for $\Gunderline_{\mathcal{H}}$ of multiplicity $k-1.$ From Remark~\ref{rem:indice-ramification-homogene}, we deduce that:
\begin{itemize}
\item If $k=2$ then the curvature of $\mathrm{Leg}\mathcal{H}$ is holomorphic on~$D.$
\smallskip

\item If $k>2$ then the curvature of  $\mathrm{Leg}\mathcal{H}$ is holomorphic on $D$ if and only if
$$0\equiv\mathrm{d}\omega\Big|_{y=x}=-k(\lambda-1)^{k-1}(\mu-1)^{k-1}x^{2k-1}(\lambda+\mu-2\lambda\,\mu)\mathrm{d}x\wedge\mathrm{d}y,$$
{\it i.e.} if and only if $\lambda$ and $\mu$ satisfy the equation $\lambda+\mu-2\lambda\,\mu=0.$
\end{itemize}
\end{eg}
\smallskip

\noindent To prove Theorem~\ref{thm:holomorphie-courbure-homogene} we need the following lemma.
\begin{lem}\label{lem:valeur-critique}
{\sl Let $f:\mathbb{P}^{1}_{\mathbb{C}}\rightarrow\mathbb{P}^{1}_{\mathbb{C}}$ be a rational map of degree $d$; $f(z)=\dfrac{a(z)}{b(z)}$ where $a$ and $b$ are polynomials without common factor and $\max(\deg a,\deg b)=d.$ Let $w_0\in\mathbb{C}$ and write $f^{-1}(w_0)=\{z_1,z_2,\ldots,z_n\}.$ Suppose that $z_i\neq\infty$ for all $i \in\{1,\ldots,n\}$ and let $\nu_i$ denote the ramification index of $f$ at the point $z_i.$ Then there exists $c\in\mathbb{C}^ {*}$ such that $a(z)=w_0b(z)+c\prod_{i=1}^{n}(z-z_i)^{\nu_i}.$
}
\end{lem}

\begin{proof}[\sl Proof]
According to~\cite[Lemma~3.9]{BM18Bull}, for every $i\in\{1,\ldots,n\}$ there exists a polynomial  $\phi_i\in\mathbb{C}[z]$ of degree~$\leq d-\nu_i$ satisfying  $\phi_i(z_i)\neq0$ and such that $a(z)=w_0b(z)+\phi_i(z)(z-z_i)^{\nu_i}.$ This implies that for all~$i,j\in\{1,\ldots,n\},$ $\phi_i(z)(z-z_i)^{\nu_i}=\phi_j(z)(z-z_j)^{\nu_j}$, so that for any $j\neq i$, $(z-z_j)^{\nu_j}$ divides $\phi_i.$ As a result, $\phi_i\in\mathbb{C}[z]$ has degree $d-\nu_i$ and writes as $\phi_i(z)=c\hspace{-3.5mm}\prod\limits_{\hspace{3.5mm}j=1,j\neq i}^{n}(z-z_j)^{\nu_j}$ for some $c\in\C^*$, hence the statement is proved.
\end{proof}

\begin{proof}[\sl Proof of Theorem~\ref{thm:holomorphie-courbure-homogene}]
Let $\delta\in\C$ be such that $b_i-a_i\delta\neq0$ for all $i=1,\ldots,n$. Up to conjugating $\omega$ by the linear transformation $(x+\delta\,y,y)$, we can assume that none of the lines $L_i=(b_i\hspace{0.2mm}y-a_i\hspace{0.2 mm}x=0)$ is vertical, {\it i.e.}  $b_i\neq0$ for all $i=1,\ldots,n.$ Setting $r_i:=\frac{a_i}{b_i} $ we have $\Gunderline_{\mathcal H}^{-1}(p_0)=\{r_1,\ldots,r_n\}$ with $\Gunderline_{\mathcal H}(z)=-\dfrac{A(1 ,z)}{B(1,z)}.$ According to Lemma~\ref{lem:valeur-critique}, there exists a constant $c\in\C^*$ such that $$-A (1,z)=p_0B(1,z)-c\prod_{i=1}^{n}(z-r_i)^{\nu_i}.$$ Moreover, the $d$-web $\Leg \mathcal{H}$ is given by the equation~(\ref{equa:LegH}); since $A,B\in\C[x,y]_d,$ this equation can then be rewritten as
\begin{align*}
0=x^d\left[A(1,p-\tfrac{q}{x})+pB(1,p-\tfrac{q}{x})\right]=x^d\Big[(p-p_0)B(1,p-\tfrac{q}{x})+c\prod_{i=1}^{n}(p-\tfrac{q}{x}-r_i)^{\nu_i}\Big], \qquad \text{with}\hspace{1mm} x=\frac{\mathrm{d}q}{\mathrm{d}p}.
\end{align*}
Set $\check{x}:=q$,\, $\check{y}:=p-p_0$\, and \,$\check{p}:=\dfrac{\mathrm{d}\check{ y}}{\mathrm{d}\check{x}}=\dfrac{1}{x}$; in these new coordinates $D=\{\check{y}=0\}$ and $\Leg\mathcal{H}$ is described by the differential equation
\begin{align*}
F(\check{x},\check{y},\check{p}):=\check{y}B(1,\check{y}+p_0-\check{p}\check{x})+c\prod_{i=1}^{n}(\check{y}+p_0-\check{p}\check{x}-r_i)^{\nu_i}=0.
\end{align*}
We have $F(\check{x},0,\check{p})=c(-\check{x})^d\prod_{i=1}^{n}(\check{p}-\varphi_{i}(\check{x}))^{\nu_i},$ where $\varphi_{i}(\check{x})=\frac{p_0-r_i}{\check{x}}.$ Note that if $\nu_i\geq2$ then $\partial_{\check{y}}F\big(\check{x},0,\varphi_{i}(\check{x})\big)=B(1,r_i)\neq0$; since $\partial_{\check{p}}F\big(\check{x},0,\varphi_{i}(\check{x})\big)\not\equiv0$ if $\nu_i=1,$ it follows that the surface $S_{\Leg\mathcal{H}}$ is smooth along $D=\{\check{y}=0\}.$ Furthermore, if $\nu_i\geq3$ then $\partial_{\check{p}}\partial_{\check{y}}F\big(\check{x},0,\varphi_{i}(\check{x})\big)=-\check{x}\partial_{y}B(1,r_i).$ Thus, by Theorem~\ref{thm-critere-holomorphie-kw}, the~curvature of $\Leg\mathcal{H}$ is holomorphic on $D=\{\check{y}=0\}$ if and only if $\sum_{i=1}^{n}(\nu_{i}-1)\varphi_{i}(\check{x})\psi_{i}\equiv0$ where, for all~$i\in\{1,\ldots,n\}$ such that $\nu_i\geq2,$
\begin{align*}
\psi_{i}=\frac{1}{\nu_{i}}
\left[
(\nu_{i}-2)\left(d+(p_0-r_i)\frac{\partial_{y}B(1,r_i)}{B(1,r_i)}\right)
+2(\nu_{i}+1)\hspace{-3.5mm}\sum\limits_{\hspace{3.5mm}j=1,j\neq i}^{n}\frac{\nu_{j}(p_0-r_j)}{r_i-r_j}
\right].
\end{align*}

\noindent We note that
\[
\hspace{-3.5mm}\sum\limits_{\hspace{3.5mm}j=1,j\neq i}^{n}\frac{\nu_{j}(p_0-r_j)}{r_i-r_j}
=d-\nu_i+(p_0-r_i)\hspace{-3.5mm}\sum\limits_{\hspace{3.5mm}j=1,j\neq i}^{n}\frac{\nu_j}{r_i-r_j}
=d-\nu_i+(p_0-r_i)\frac{f_{i}^{'}(r_i)}{f_{i}(r_i)},
\]
where
$
f_{i}(z):=c\hspace{-3.5mm}\prod\limits_{\hspace{3.5mm}j=1,j\neq i}^{n}(z-r_j)^{\nu_j}
=\dfrac{A(1,z)+p_0B(1,z)}{(z-r_i)^{\nu_i}}
=\dfrac{P_i(1,z;r_i,1)}{B(1,r_i)}.
$
Therefore
\begin{align*}
&\hspace{1cm}\hspace{-3.5mm}\sum\limits_{\hspace{3.5mm}j=1,j\neq i}^{n}\frac{\nu_{j}(p_0-r_j)}{r_i-r_j}
=d-\nu_i+(p_0-r_i)\frac{\partial_{y}P_i(1,r_i;r_i,1)}{P_i(1,r_i;r_i,1)}
=\frac{\left|\begin{array}{cc}
\partial_{x}P_i(1,r_i;r_i,1) &  A(1,r_i)
\vspace{2mm}
\\
\partial_{y}P_i(1,r_i;r_i,1) &  B(1,r_i)
\end{array} \right|}{B(1,r_i)P_i(1,r_i;r_i,1)},
\end{align*}
because $p_0=\Gunderline_{\mathcal H}(r_i)=-\dfrac{A(1,r_i)}{B(1,r_i)}$\, and \,$(d-\nu_i)P_i(1,r_i;r_i,1)=\partial_{x}P_i(1,r_i;r_i,1)+r_i\partial_{y}P_i(1,r_i;r_i,1)$ (\textsc{Euler}'s identity).

\noindent On the other hand, let us fix $i\in\{1,\ldots,n\}$ such that $\nu_i\geq2$; from the equalities $p_0=\Gunderline_{\mathcal H}(r_i)$ and $\Gunderline_{\mathcal H}^{'}(r_i)=0$ we deduce that $p_0\partial_{y}B(1,r_i)=-\partial_{y}A(1,r_i),$ so that
\[
dB(1,r_i)+(p_0-r_i)\partial_{y}B(1,r_i)=dB(1,r_i)-r_i\partial_{y}B(1,r_i)-\partial_{y}A(1,r_i)
=\partial_{x}B(1,r_i)-\partial_{y}A(1,r_i),
\]
thanks to \textsc{Euler}'s identity.
\smallskip

\noindent It follows that for all $i\in\{1,\ldots,n\}$ such that $\nu_i\geq2,$ $\psi_{i}=\dfrac{Q_{i}(1,r_i;r_i ,1)}{\nu_{i}P_{i}(1,r_i;r_i,1)B(1,r_i)}.$ As a consequence, $K(\Leg\mathcal{H})$~is~ holomorphic on $D=\{\check{y}=0\}$ if~and~only~if
\begin{align*}
\frac{1}{\check{x}}\sum_{i=1}^{n}\left(1-\frac{1}{\nu_{i}}\right)\frac{(p_0-r_i)Q_{i}(1,r_i;r_i,1)}{P_{i}(1,r_i;r_i,1)B(1,r_i)}\equiv0,
\end{align*}
which ends the proof of the theorem.
\end{proof}

\begin{cor}\label{cor:holomorphie-droite-inflex-nu-1}
{\sl Let $\mathcal{H}$ be a homogeneous foliation of degree $d\geq3$ on $\pp$ defined by the $1$-form $$\omega=A(x,y)\mathrm{d}x+B(x,y)\mathrm{d}y,\quad A,B\in\mathbb{C}[x,y]_d, \hspace{2mm}\gcd(A,B)=1.$$ Assume that $\mathcal{H}$ possesses a transverse inflection line $T=(ax+by=0)$ of order $\nu-1$. Suppose moreover that $[-a:b]\in\mathbb{P}^{1}_{\mathbb{C}}$ is the only non-fixed critical point of $\Gunderline_{\mathcal H}$ in its fiber $\Gunderline_{\mathcal H}^{-1}(\Gunderline_{\mathcal H}([-a:b])).$ Then the curvature of $\mathrm{Leg}\mathcal{H}$ is holomorphic on $T'=\mathcal{G}_{\mathcal{H}}(T)$ if and only if $Q(b,-a\hspace{0.2mm};a,b)=0$,~where
\begin{Small}
\begin{align*}
Q(x,y;a,b):=(\nu-2)\left(\dfrac{\partial{B}}{\partial{x}}-\dfrac{\partial{A}}{\partial{y}}\right)P(x,y;a,b)+2(\nu+1)
\left|\begin{array}{cc}
\dfrac{\partial{P}}{\partial{x}} &  A(x,y)
\vspace{2mm}
\\
\dfrac{\partial{P}}{\partial{y}} &  B(x,y)
\end{array} \right|
\quad{\fontsize{11}{11pt}\text{and}}\quad
P(x,y;a,b):=\frac{\left|
\begin{array}{cc}
A(x,y)  &  A(b,-a)
\\
B(x,y)  &  B(b,-a)
\end{array}
\right|}{(ax+by)^{\nu}}.
\end{align*}
\end{Small}
}
\end{cor}

\begin{rem}\label{rem:holomorphie-droite-inflex-minimale-maximale}
When the line $T=(ax+by=0)$ is of minimal inflection order $1$ ({\it i.e.} if $\nu=2$) and under the more restrictive hypothesis that the point $[-a:b]$ is the only critical point of $\Gunderline_{\mathcal H}$ in its fiber, we recover \cite[Theorem~3.5]{BM18Bull}. When $T$ is of maximal inflection order $d-1$ ({\it i.e.} if $\nu=d$), we recover~\cite[Theorem~3.8]{BM18Bull}.
\end{rem}


\section{\textsc{Galois} homogeneous foliations having a flat \textsc{Legendre} transform}\label{sec:feuill-homog-Galois-plat}
\vspace{2mm}

\noindent Following~\cite[Definition~6.16]{BFMN16} a foliation $\F$ of degree $d$ on $\pp$ is said to be \textsc{Galois} if there is a \textsc{Zariski} open subset $U$ of $\pp$ such that the \textsc{Gauss} map $\mathcal{G}_{\mathcal{F}}\hspace{1mm}\colon\pp\dashrightarrow \pd$, defined by $m\not\in\Sing\F\mapsto\mathrm{T}^{\mathbb{P}}_{m}\F$, induces a \textsc{Galois} covering from $U$ onto $\mathcal{G}_{\mathcal{F}}(U)$, necessarily of degree $d$. This is equivalent to the existence of a subgroup $G$ of order $d$ of the group $\mathrm{Bir}(\pp)$ of birational transformations of $\pp$ such that for all $\gamma\in G$, we have $\mathcal{G}_{\mathcal{F}}\circ\gamma=\mathcal{G}_{\mathcal{F}}.$

\noindent In particular, if $\F$ is homogeneous, then its associated map $\Gunderline_{\mathcal{F}}\hspace{1mm}\colon\mathbb{P}^{1}_{\mathbb{C}}\rightarrow \mathbb{P}^{1}_{\mathbb{C}}$ is a ramified covering of degree~$d$. Moreover, $\F$ is \textsc{Galois} if and only if $\Gunderline_{\mathcal{F}}$ is \textsc{Galois} (\cite[Proposition~6.19]{BFMN16}), or equivalently, if~and~only~if $\Gunderline_{\mathcal{F}}$ has the same ramification indices at all the points of the same fiber (\cite[Theorem~A]{BFMN16}).
\smallskip

\noindent Let us recall the following result classifying the ramified \textsc{Galois} coverings of the \textsc{Riemann} sphere by itself, due to \textsc{Klein}~\cite[Part I,~Chapter~II]{Kle03} (\emph{see} also \cite[Theorem~4.18]{BFMN16}).
\begin{thm}\label{thm:Klein}
{\sl Let $f:\mathbb{P}^{1}_{\mathbb{C}}\rightarrow\mathbb{P}^{1}_{\mathbb{C}}$ be a ramified \textsc{Galois} covering of degree $d. $ Up to the left-right-action of $(\mathrm{Aut}(\mathbb{P}^{1}_{\mathbb{C}}))^2$, $f$ is of one of the following types
\begin{itemize}
\item [\texttt{1.}] $f_1=z^d$;
\smallskip
\item [\texttt{2.}] $f_2=\frac{(z^k+1)^2}{4z^k}$ if $d$ is even, $d=2k$;
\smallskip
\item [\texttt{3.}] $f_3=\left(\frac{z^4+2\mathrm{i}\sqrt{3}z^2+1}{z^4-2\mathrm{i}\sqrt{3}z^2+1}\right)^3$ if $d=12$;
\smallskip
\item [\texttt{4.}] $f_4=\frac{\left(z^8+14z^4+1\right)^3}{108z^4\left(z^4-1\right)^4}$ if $d=24$;
\smallskip
\item [\texttt{5.}] $f_5=\frac{\left(z^{20}-228z^{15}+494z^{10}+228z^5+1\right)^3}{-1728z^5\left(z^{10}+11z^5-1\right)^5}$ if $d=60$.
\end{itemize}
\smallskip

\noindent Moreover, the \textsc{Galois} group of $f$ is cyclic if and only if $f$ is left-right conjugate to $f_1.$}
\end{thm}

\begin{defin}
\noindent Let $f:\mathbb{P}^{1}_{\mathbb{C}}\rightarrow\mathbb{P}^{1}_{\mathbb{C}}$ be a rational map of degree $d.$ We call {\sl associated foliation} to~$f$ the homogeneous foliation $\mathscr{H}(f)$ of $\pp$ whose associated rational map $\Gunderline_{\mathscr{H}(f )}$ is precisely $f.$
\end{defin}

\noindent Note that if $f$ is defined by $f([x:y])=[A(x,y):B(x,y)],$ where $A,B\in\mathbb{C}[x ,y]_d$\, and $\mathrm{gcd}(A,B)=1,$ then $\mathscr{H}(f)$ is given by the $1$-form $\omega=A(y, x)\mathrm{d}x-B(y,x)\mathrm{d}y.$
\smallskip

\noindent Theorem~\ref{thm:Klein} translates in terms of homogeneous foliations as follows:
\begin{thm}
{\sl Let $\mathcal{H}$ be a \textsc{Galois} homogeneous foliation on $\pp.$ Then there exist $i\in\{1,\ldots,5\}$ and $\ell,\rho\in\mathrm{Aut}(\mathbb{P}^{1}_{\mathbb{C}})$~such~that~$\mathcal{H}=\mathscr{H}(\ell\circ f_i\circ\rho).$
}
\end{thm}

\noindent The following theorem is the main result of this section.
\begin{thm}\label{thm:feuill-homog-Galois-plat}
{\sl Let $\mathcal{H}$ be a \textsc{Galois} homogeneous  foliation of degree $d\geq3$ on $\pp.$ Denote by $\mathrm{Gal}(\Gunderline_{\mathcal H})$ the \textsc{Galois} group of the covering $\Gunderline_{\mathcal H}.$ We have the following dichotomy:
\begin{itemize}
\item If $\mathrm{Gal}(\Gunderline_{\mathcal H})$ is cyclic, then the $d$-web $\Leg\mathcal{H}$ is flat if and only if $\mathcal{H} $ is linearly conjugate to one of the two foliations $\mathcal{H}_{1}^{d}$ and $\mathcal{H}_{2}^{d}$ defined respectively by the $1$-forms
\begin{align*}
\omega_1^{\hspace{0.2mm}d}=y^d\mathrm{d}x-x^d\mathrm{d}y
\qquad\qquad\text{and}\qquad\qquad
\omega_2^{\hspace{0.2mm}d}=x^d\mathrm{d}x-y^d\mathrm{d}y.
\end{align*}
\item If $\mathrm{Gal}(\Gunderline_{\mathcal H})$ is non-cyclic, then the $d$-web $\Leg\mathcal{H}$ is flat.
\end{itemize}
}
\end{thm}

\noindent To prove this theorem, we need the following lemma.
\begin{lem}\label{lem:holomorphie-courbure-Leg-Hh}
{\sl Let $f:\mathbb{P}^{1}_{\mathbb{C}}\rightarrow\mathbb{P}^{1}_{\mathbb{C}}$ be a rational map of degree $d$ defined, for any $[x:y]\in\sph,$ by $$f([x:y])=[A(x,y):B(x,y)],\quad A,B\in\mathbb{C}[x,y]_d,\hspace{2mm }\gcd(A,B)=1.$$ Let $p_0\in\C\cup\{\infty\}$ be a critical value of $f$ and write $f^{-1}(p_0)=\{[a_1:b_1],\ldots,[a_n: b_n]\}.$ Suppose that the ramification indices of $f$ at the points $[a_i:b_i]$ are all equal to each other and let $\nu$ be their common value. For~$h\in \mathrm{Aut}(\mathbb{P}^{1}_{\C})$ denote by $\mathcal{H}_h=\mathscr{H}(h\circ f)$ the homogeneous~foliation~associated to the rational map $h\circ f.$ Let $(p,q)$ be the affine chart of $\pd$ corresponding to the line $\{y=px-q\}\subset{\mathbb{P}^{2}_{\mathbb{C}}}$ and let $D_h:=\{p=h(p_0)\}\subset \Delta(\Leg\mathcal{H}_h)$.
\begin{itemize}
\item [$\bullet$] If $\nu=2$ then the curvature of $\Leg\mathcal{H}_h$ is holomorphic on $D_h$ for all $h\in \mathrm{Aut}(\mathbb{P}^{1}_{\C}).$
\smallskip
\item [$\bullet$] If $\nu\geq3$ and $p_0\in\C$ then the curvature of  $\Leg\mathcal{H}_h$ is holomorphic on $D_h$ for all $h\in\mathrm{Aut}(\mathbb{P}^{1}_{\C})$ if~and~only~if
\begin{align}\label{equa:holomorphie-courbure-Dh-nu-3-fini}
\sum_{i=1}^{n}\frac{b_i\partial_{x}B(a_i,b_i)}{B(a_i,b_i)}=0,
&&
\sum_{i=1}^{n}\frac{b_i\partial_{y}B(a_i,b_i)-a_i\partial_{x}B(a_i,b_i)}{B(a_i,b_i)}=0,
&&
\sum_{i=1}^{n}\frac{a_i\partial_{y}B(a_i,b_i)}{B(a_i,b_i)}=0.
\end{align}
\item [$\bullet$] If $\nu\geq3$ and $p_0=\infty$ then the curvature of $\Leg\mathcal{H}_h$ is holomorphic on $D_h$ for all $h\in\mathrm{Aut}(\mathbb{P}^{1}_{\C})$ if~and~only~if
\begin{align}\label{equa:holomorphie-courbure-Dh-nu-3-infini}
\sum_{i=1}^{n}\frac{b_i\partial_{x}A(a_i,b_i)}{A(a_i,b_i)}=0,
&&
\sum_{i=1}^{n}\frac{b_i\partial_{y}A(a_i,b_i)-a_i\partial_{x}A(a_i,b_i)}{A(a_i,b_i)}=0,
&&
\sum_{i=1}^{n}\frac{a_i\partial_{y}A(a_i,b_i)}{A(a_i,b_i)}=0.
\end{align}
\end{itemize}
}
\end{lem}

\begin{proof}
Let $h:\mathbb{P}^{1}_{\mathbb{C}}\rightarrow\mathbb{P}^{1}_{\mathbb{C}}$ be an automorphism of $\mathbb{P}^{1}_{\mathbb{C}}$; $h(z)=\dfrac{\alpha z+\beta}{\gamma z+\delta}$ where $\alpha,\beta,\gamma,\delta\in\C$ with $\alpha\delta-\beta\gamma\neq0.$ Then the foliation $\mathcal{H}_h$ is given by
\[
\omega_h=\big(\alpha\,A(y,x)+\beta\,B(y,x)\big)\mathrm{d}x-\big(\gamma\,A(y,x)+\delta\,B(y,x)\big)\mathrm{d}y.
\]
Moreover we have
\[
(h\circ f)^{-1}(h(p_0))=f^{-1}(p_0)=\{[a_1:b_1],\ldots,[a_n:b_n]\};
\]
since by hypothesis the ramification indices of $f$ at the points $[a_i:b_i]$ are all equal to each other and equal~to~$\nu$, the same is true for the ramification indices of $h\circ f $ at these points, because $h\in \mathrm{Aut}(\mathbb{P}^{1}_{\C}).$ According to Remark~\ref{rem:indice-ramification-homogene}, it follows that:
\begin{itemize}
\item [\textbf{\textit{i.}}] If $\nu=2$ then $K(\Leg\mathcal{H}_h)$ is holomorphic on $D_h$ for all $h\in \mathrm{Aut}(\mathbb{P}^{1}_{\C}).$
\smallskip
\item [\textbf{\textit{ii.}}] If $\nu\geq3$ then $K(\Leg\mathcal{H}_h)$ is holomorphic on  $D_h$ for all $h\in \mathrm{Aut}(\mathbb{P}^{1}_{\C})$~if~and~only~if
\begin{align}\label{equa:holomorphie-courbure-Dh-nu-3}
\sum_{i=1}^{n}\frac{\Big(h(p_0)b_i-a_i\Big)\Big(\alpha\partial_{x}A(a_i,b_i)+\beta\partial_{x}B(a_i,b_i)+\gamma\partial_{y}A(a_i,b_i)+\delta\partial_{y}B(a_i,b_i)\Big)}{\gamma\,A(a_i,b_i)+\delta\,B(a_i,b_i)}=0.
\end{align}
\item [\textbf{\textit{ii.1.}}] If $p_0\in\C$ then, from $f([a_i:b_i])=[p_0:1]$ and the fact that $[a_i:b_i]$ are critical points of  $f,$ we deduce the equalities
$A(a_i,b_i)=p_0B(a_i,b_i)$, $\partial_{x}A(a_i,b_i)=p_0\partial_{x}B(a_i,b_i)$ and $\partial_{y}A(a_i,b_i)=p_0\partial_{y}B(a_i,b_i),$ so that (\ref{equa:holomorphie-courbure-Dh-nu-3}) can be rewritten as
\begin{align*}
h(p_0)^2\sum_{i=1}^{n}\frac{b_i\partial_{x}B(a_i,b_i)}{B(a_i,b_i)}+h(p_0)\sum_{i=1}^{n}\frac{b_i\partial_{y}B(a_i,b_i)-a_i\partial_{x}B(a_i,b_i)}{B(a_i,b_i)}-\sum_{i=1}^{n}\frac{a_i\partial_{y}B(a_i,b_i)}{B(a_i,b_i)}=0.
\end{align*}
As a result, $K(\Leg\mathcal{H}_h)$ is holomorphic on $D_h$ for all $h\in \mathrm{Aut}(\mathbb{P}^{1}_{\C})$ if and only if the system~(\ref{equa:holomorphie-courbure-Dh-nu-3-fini}) is satisfied.
\smallskip

\item [\textbf{\textit{ii.2.}}] If $p_0=\infty$ then $B(a_i,b_i)=\partial_{x}B(a_i,b_i)=\partial_{y}B(a_i,b_i)=0$ and (\ref{equa:holomorphie-courbure-Dh-nu-3}) becomes
\begin{align*}
h(p_0)^2\sum_{i=1}^{n}\frac{b_i\partial_{x}A(a_i,b_i)}{A(a_i,b_i)}+h(p_0)\sum_{i=1}^{n}\frac{b_i\partial_{y}A(a_i,b_i)-a_i\partial_{x}A(a_i,b_i)}{A(a_i,b_i)}-\sum_{i=1}^{n}\frac{a_i\partial_{y}A(a_i,b_i)}{A(a_i,b_i)}=0.
\end{align*}
As a consequence, $K(\Leg\mathcal{H}_h)$ is holomorphic on $D_h$ for all $h\in \mathrm{Aut}(\mathbb{P}^{1}_{\C})$ if and only if the system~(\ref{equa:holomorphie-courbure-Dh-nu-3-infini}) is satisfied.
\end{itemize}
Hence the lemma is proved.
\end{proof}

\begin{proof}[\sl Proof of Theorem~\ref{thm:feuill-homog-Galois-plat}]
\textbf{\textit{i.}} Suppose that $\mathrm{Gal}(\Gunderline_{\mathcal H})$ is cyclic. Then, by Theorem~\ref{thm:Klein}, $\Gunderline_{\mathcal H}$ is left-right conjugate to $f_1=z^d$. Since $f_1$ has exactly two critical points (namely $0$ and $\infty$), the same is true for $\Gunderline_{\mathcal H}.$ This~implies, according to \cite[Proposition ~4.1]{BM18Bull}, that the $d$-web $\Leg\mathcal{H}$ is flat if~and~only~if $\mathcal{H}$ is linearly conjugate to one of the two foliations $ \mathcal{H}_{1}^{d}$, $\mathcal{H}_{2}^{d}.$
\vspace{2mm}

\textbf{\textit{ii.}} Suppose that $\mathrm{Gal}(\Gunderline_{\mathcal H})$ is non-cyclic. According to Theorem~\ref{thm:Klein}, there exist $i\in\{2,\ldots,5\}$ and $\ell,\rho\in\mathrm{Aut}(\mathbb {P}^{1}_{\mathbb{C}})$ such that $\Gunderline_{\mathcal H}=\ell\circ f_i\circ\rho$ and therefore $\mathcal{H}=\mathscr{H}(\ell\circ f_i\circ\rho).$ In particular, there exist $i\in\{2,\ldots,5\}$ and~$h\in\mathrm{Aut}(\mathbb {P}^{1}_{\mathbb{C}})$ such that $\mathcal{H}$ is linearly conjugate to the foliation $\mathcal{H}_{h}^{(i)}:=\mathscr{H}(h\circ f_i)$; indeed, it suffices to take $h=\rho\circ\ell$, because $h\circ f_i=\rho\circ(\ell\circ f_i\circ\rho)\circ\rho^{-1}. $ To show that the $d$-web $\Leg\mathcal{H}$ is flat, it suffices therefore to show that for all $i\in\{2,\ldots,5\}$ and~all~$ h\in\mathrm{Aut}(\mathbb{P}^{1}_{\mathbb{C}}),$ the $d$-web $\Leg\mathcal{H}_{h}^{( i)}$ is flat. Now, for all $i\in\{2,\ldots,5\}$, the map $f_i$ being a ramified \textsc{Galois} covering of $\sph$ by itself, \cite[Theorem~A]{BFMN16} implies that the ramification indices of $f_i$ at the points of the same fiber $f_{i}^{-1}(p_0)$ have the same value, which we will denote by $\nu(f_i,p_0).$ Thanks to~\cite[Theorem~3.1]{BM18Bull}, it suffices again to apply Lemma~\ref{lem:holomorphie-courbure-Leg-Hh} to each of the $f_i$ and to show that for every critical value $p_0\in\sph$ of $f_i,$ the curvature of $\Leg\mathcal{H}_{h}^{(i)}$ is holomorphic on the component $D_{h}^{(i)} (p_0):=\{p=h(p_0)\}$ of~$\Delta(\Leg\mathcal{H}_{h}^{(i)})$ for all $h\in \mathrm{Aut}(\mathbb{P}^{1}_{\C}).$
\vspace{2mm}

\noindent First of all, a straightforward computation shows that each of the $f_i,i=2,\ldots,5,$ has as critical values $0$, $1$ and~$\infty.$
\vspace{2mm}

\noindent The~case~of~the~critical~value $p_0=1$ is immediate. Indeed, it is easy to verify that for all $i\in\{2,\ldots,5\},$ $\nu(f_i,1)=2,$ so that the curvature of $\Leg\mathcal{H}_{h}^{(i)}$ is holomorphic on $D_{h}^{(i)}(1)$ for all $i\in\{2,\ldots,5\}$ and all $h\in \mathrm{Aut}(\mathbb{P}^{1}_{\C})$ (Lemma~\ref{lem:holomorphie-courbure-Leg-Hh}).
\vspace{2mm}

\noindent The case where $i=2$ and $p_0=0$ is also immediate. Indeed, we have $\nu(f_2,0)=2,$ which implies that $K(\Leg\mathcal{H}_{h}^{(2)})$ is holomorphic on $D_{h }^{(2)}(0)$ for all $h\in\mathrm{Aut}(\mathbb{P}^{1}_{\C}).$
\vspace{2mm}

\noindent Let us consider the case where $i=2$ and $p_0=\infty.$ The map $f_2$ is defined in homogeneous coordinates by
\[
\hspace{1.5cm}f_2\hspace{1mm}\colon[x:y]\mapsto [A_2(x,y):B_2(x,y)],\quad\text{where } A_2(x,y)=(x^k+y^k)^2\hspace{1mm} \text{and } B_2(x,y)=4x^ky^k.
\]
Moreover, the fiber $f_{2}^{-1}(\infty)$ consists of the two points $0=[0:1]$ and $\infty=[1:0]$; in particular $\nu(f_2,\infty)=k.$ If~$k=2$~then $K(\Leg\mathcal{H}_{h}^{(2)})$ is holomorphic on $ D_{h}^{(2)}(\infty)$ for all $h\in\mathrm{Aut}(\mathbb{P}^{1}_{\C}).$ Suppose $k\geq3. $ We have
\begin{tiny}
\begin{align*}
&\sum_{[a:b]\in f_{2}^{-1}(\infty)}\frac{b\partial_{x}A_{2}(a,b)}{A_{2}(a,b)}=\frac{\partial_{x}A_{2}(0,1)}{A_{2}(0,1)}=0,&&
\sum_{[a:b]\in f_{2}^{-1}(\infty)}\frac{b\partial_{y}A_{2}(a,b)-a\partial_{x}A_{2}(a,b)}{A_{2}(a,b)}=
\frac{\partial_{y}A_{2}(0,1)}{A_{2}(0,1)}-\frac{\partial_{x}A_{2}(1,0)}{A_{2}(1,0)}=0,&&
\sum_{[a:b]\in f_{2}^{-1}(\infty)}\frac{a\partial_{y}A_{2}(a,b)}{A_{2}(a,b)}=\frac{\partial_{y}A_{2}(1,0)}{A_{2}(1,0)}=0\hspace{0.5mm};
\end{align*}
\end{tiny}
\hspace{-0.9mm}it follows, by Lemma~\ref{lem:holomorphie-courbure-Leg-Hh}, that $K(\Leg\mathcal{H}_{h}^{(2)})$ is holomorphic on $D_{h}^{(2)}(\infty)$ for all $h\in\mathrm{Aut}(\mathbb{P}^{1}_{\C}).$
\vspace{2mm}

\noindent Let us study the case where $i=5$ and $p_0=0.$ Consider the polynomials
\begin{align*}
P(w)=w^4-228w^3+494w^2+228w+1
\qquad\text{and}\qquad
Q(w)=-\sqrt[5]{1728}(w^2+11w-1)\hspace{0.5mm};
\end{align*}
the map $f_5$ is given, for any $[x:y]\in\sph$, by  $f_5([x:y])=[A_5(x,y):B_5(x,y)],$ where
\begin{small}
\begin{align*}
A_5(x,y)=\left(y^{20}P\left(\frac{x^5}{y^5}\right)\right)^3
\qquad\text{and}\qquad
B_5(x,y)=\left(xy^{11}Q\left(\frac{x^5}{y^5}\right)\right)^5.
\end{align*}
\end{small}
\hspace{-1mm}The polynomial $P(w)$ has as roots the real numbers
\begin{tiny}
\begin{align*}
w_1=57-25\sqrt{5}+5\sqrt{255-114\sqrt{5}},&&
w_2=57-25\sqrt{5}-5\sqrt{255-114\sqrt{5}},&&
w_3=57+25\sqrt{5}+5\sqrt{255+114\sqrt{5}},&&
w_4=57+25\sqrt{5}-5\sqrt{255+114\sqrt{5}}\hspace{0.5mm};
\end{align*}
\end{tiny}
\hspace{-1mm}by setting $\zeta=\exp(\frac{2\mathrm{i}\pi}{5})$ and $u_j=\sqrt[5]{w_j}\in\R$, $j=1,\ldots,4,$ we have
\[
f_{5}^{-1}(0)=\Big\{[\zeta^{l}u_j:1]\hspace{1.5mm}\big\vert\hspace{1.5mm} j=1,\ldots,4,\hspace{1mm}l=0,\ldots,4\Big\}.
\]
In particular, $f_{5}^{-1}(0)$ has cardinality $20$ and therefore $\nu(f_5,0)=60/20=3.$ Furthermore, by a straightforward computation, we obtain the following equalities
\begin{Small}
\begin{align*}
&
\frac{b\partial_{x}B_{5}(a,b)}{B_{5}(a,b)}\Big|_{(a,b)=(\zeta^{l}u_j,1)}
=5\zeta^{5-l}\left(\frac{1}{u_{j}}+\frac{5w_{j}Q\hspace{0.2mm}'(w_{j})}{u_{j}Q(w_{j})}\right),
&&
\frac{a\partial_{y}B_{5}(a,b)}{B_{5}(a,b)}\Big|_{(a,b)=(\zeta^{l}u_j,1)}
=5\zeta^{l}u_{j}\left(11-\frac{5w_{j}Q\hspace{0.2mm}'(w_{j})}{Q(w_{j})}\right),\\
&
\frac{b\partial_{y}B_{5}(a,b)-a\partial_{x}B_{5}(a,b)}{B_{5}(a,b)}\Big|_{(a,b)=(\zeta^{l}u_j,1)}
=g(w_j),
\end{align*}
\end{Small}
\hspace{-1mm}where $g\hspace{1mm}\colon x\mapsto -\frac{50(x^2+1)}{x^2+11x-1},$ so that
\begin{Small}
\begin{align*}
\sum_{j=1}^{4}\sum_{l=0}^{4}\frac{b\partial_{x}B_{5}(a,b)}{B_{5}(a,b)}\Big|_{(a,b)=(\zeta^{l}u_j,1)}=0,&&
\sum_{j=1}^{4}\sum_{l=0}^{4}\frac{b\partial_{y}B_{5}(a,b)-a\partial_{x}B_{5}(a,b)}{B_{5}(a,b)}\Big|_{(a,b)=(\zeta^{l}u_j,1)}=0,&&
\sum_{j=1}^{4}\sum_{l=0}^{4}\frac{a\partial_{y}B_{5}(a,b)}{B_{5}(a,b)}\Big|_{(a,b)=(\zeta^{l}u_j,1)}=0,
\end{align*}
\end{Small}
\hspace{-1mm}because $\sum_{l=0}^{4}\zeta^{l}=\sum_{l=0}^{4}\zeta^{5-l}=0$\, and \,$\sum_{j=1}^{4}g(w_j)=0.$ Thus, we deduce from Lemma~\ref{lem:holomorphie-courbure-Leg-Hh} that $K(\Leg\mathcal{H}_{h}^{(5)})$ is holomorphic on~$D_{h}^{(5)}(0)$ for all $h\in\mathrm{Aut}(\mathbb{P}^{1}_{\C})$.
\vspace{2mm}

\noindent Let us examine the case where $i=5$ and $p_0=\infty.$ Set $\widetilde{w}_1=\frac{-11+5\sqrt{5}}{2}$,\, $\widetilde{ w}_2=\frac{-11-5\sqrt{5}}{2}$,\, $\widetilde{u}_1=\frac{-1+\sqrt{5}}{2}$\, and \,$\widetilde{u}_2=\frac{-1-\sqrt{5}}{2}$ (the~$\widetilde{w}_j$~are the two roots of $Q(w)$ and $\widetilde{u}_j=\sqrt[5]{\widetilde{w}_j}$\hspace{0.2mm}). Then
\[
f_{5}^{-1}(\infty)=\Big\{[0:1],\,[1:0],\,[\zeta^{l}\widetilde{u}_j:1]\hspace{1.5mm}\big\vert\hspace{1.5mm} j=1,2,\hspace{1mm}l=0,\ldots,4\Big\}\hspace{0.5mm};
\]
in particular, $\#\hspace{0.1mm}f_{5}^{-1}(\infty)=12$ and consequently $\nu(f_5,\infty)=60/12=5.$ Moreover, a straightforward computation leads to
\begin{tiny}
\begin{align*}
&
\frac{b\partial_{x}A_{5}(a,b)}{A_{5}(a,b)}\Big|_{(a,b)=(0,1)}=0,
&&
\frac{a\partial_{y}A_{5}(a,b)}{A_{5}(a,b)}\Big|_{(a,b)=(0,1)}
=0,
&&
\frac{b\partial_{y}A_{5}(a,b)-a\partial_{x}A_{5}(a,b)}{A_{5}(a,b)}\Big|_{(a,b)=(0,1)}
=60,\\
&
\frac{b\partial_{x}A_{5}(a,b)}{A_{5}(a,b)}\Big|_{(a,b)=(1,0)}
=0,
&&
\frac{a\partial_{y}A_{5}(a,b)}{A_{5}(a,b)}\Big|_{(a,b)=(1,0)}
=0,
&&
\frac{b\partial_{y}A_{5}(a,b)-a\partial_{x}A_{5}(a,b)}{A_{5}(a,b)}\Big|_{(a,b)=(1,0)}
=-60,\\
&
\frac{b\partial_{x}A_{5}(a,b)}{A_{5}(a,b)}\Big|_{(a,b)=(\zeta^{l}\widetilde{u}_j,1)}
=\frac{15\zeta^{5-l}\widetilde{w}_{j}P\hspace{0.2mm}'(\widetilde{w}_{j})}{\widetilde{u}_{j}P(\widetilde{w}_{j})},
&&
\frac{a\partial_{y}A_{5}(a,b)}{A_{5}(a,b)}\Big|_{(a,b)=(\zeta^{l}\widetilde{u}_j,1)}
=15\zeta^{l}\widetilde{u}_j\left(4-\frac{\widetilde{w}_{j}P\hspace{0.2mm}'(\widetilde{w}_{j})}{P(\widetilde{w}_{j})}\right),
&&
\frac{b\partial_{y}A_{5}(a,b)-a\partial_{x}A_{5}(a,b)}{A_{5}(a,b)}\Big|_{(a,b)=(\zeta^{l}\widetilde{u}_j,1)}
=\widetilde{g}(\widetilde{w}_{j}),
\end{align*}
\end{tiny}
\hspace{-1mm}where $\widetilde{g}\hspace{1mm}\colon x\mapsto -\frac{60(x^4-114\,x^3-114\,x-1)}{x^4-228\,x^3+494\,x^2+228\,x+1}.$ Therefore, we have
\begin{Small}
\begin{align*}
&\sum_{[a:b]\in f_{5}^{-1}(\infty)}\frac{b\partial_{x}A_{5}(a,b)}{A_{5}(a,b)}
=\sum_{j=1}^{2}\frac{15\widetilde{w}_{j}P\hspace{0.2mm}'(\widetilde{w}_{j})}{\widetilde{u}_{j}P(\widetilde{w}_{j})}\sum_{l=0}^{4}\zeta^{5-l}
=0,&&
\sum_{[a:b]\in f_{5}^{-1}(\infty)}\frac{a\partial_{y}A_{5}(a,b)}{A_{5}(a,b)}
=\sum_{j=1}^{2}15\widetilde{u}_j\left(4-\frac{\widetilde{w}_{j}P\hspace{0.2mm}'(\widetilde{w}_{j})}{P(\widetilde{w}_{j})}\right)\sum_{l=0}^{4}\zeta^{l}=0,
\\
&\sum_{[a:b]\in f_{5}^{-1}(\infty)}
\frac{b\partial_{y}A_{5}(a,b)-a\partial_{x}A_{5}(a,b)}{A_{5}(a,b)}=5\sum_{j=1}^{2}\widetilde{g}(\widetilde{w}_j)=0.
\end{align*}
\end{Small}
\hspace{-1mm}According to Lemma~\ref{lem:holomorphie-courbure-Leg-Hh}, it follows that $K(\Leg\mathcal{H}_{h}^{(5)})$ is holomorphic on~$D_{h}^{(5)}(\infty)$ for all $h\in\mathrm{Aut}(\mathbb{P}^{1}_{\C})$.
\vspace{1.5mm}

\noindent The remaining cases (those where $i\in\{3,4\}$ and $p_0\in\{0,\infty\}$) are treated similarly.
\end{proof}

\begin{rem}
For $d\geq3,$ denote by $\mathbf{FP}(d)$ the algebraic set consisting of foliations of degree $d$ on $\pp$ with a flat \textsc{Legendre} transform. In~\cite[Theorem~D]{BM21Four}, we showed that $\mathbf{FP}(3)$ has exactly twelve irreducible components, each of them is rigid in the sense that it is the closure of the orbit under the action of $\mathrm{Aut}(\pp)$ of a foliation on $\pp.$ Theorem~\ref{thm:feuill-homog-Galois-plat} shows that in any even degree $d$ the algebraic set $\mathbf{FP}(d)$ always contains non-rigid irreducible components.
\end{rem}


\end{document}